\documentclass[11pt]{article}

\usepackage[margin=1.1in]{geometry}
\usepackage[T1]{fontenc}
\usepackage[utf8]{inputenc}
\usepackage{microtype}
\usepackage{amsmath,amssymb,amsthm,mathtools}
\usepackage{thmtools}
\usepackage{bm}
\usepackage{booktabs}
\usepackage{multirow}
\usepackage{enumitem}
\usepackage{graphicx}
\usepackage{tikz}
\usetikzlibrary{positioning,arrows.meta,calc}
\usepackage[numbers,sort&compress]{natbib}
\usepackage[colorlinks=true,linkcolor=blue!60!black,citecolor=blue!60!black,urlcolor=blue!60!black]{hyperref}
\usepackage[capitalize,noabbrev]{cleveref}
\usepackage[ruled,vlined,linesnumbered]{algorithm2e}
\crefname{algocf}{Algorithm}{Algorithms}
\Crefname{algocf}{Algorithm}{Algorithms}
\usepackage{xcolor}

\ifdefined\JOTAmode
  \spdefaulttheorem{assumption}{Assumption}{\bfseries}{\rmfamily}
\else
  \declaretheorem[style=plain,numberwithin=section,name=Theorem]{theorem}
  \declaretheorem[style=plain,sibling=theorem,name=Lemma]{lemma}
  \declaretheorem[style=plain,sibling=theorem,name=Proposition]{proposition}
  \declaretheorem[style=plain,sibling=theorem,name=Corollary]{corollary}
  \declaretheorem[style=definition,sibling=theorem,name=Definition]{definition}
  \declaretheorem[style=definition,sibling=theorem,name=Assumption]{assumption}
  
  \declaretheorem[style=remark,sibling=theorem,name=Remark]{remark}

  \AtBeginDocument{%
  }
\fi

\crefname{appendix}{Appendix}{Appendices}
\Crefname{appendix}{Appendix}{Appendices}

\crefname{theorem}{Theorem}{Theorems}
\Crefname{theorem}{Theorem}{Theorems}
\crefname{lemma}{Lemma}{Lemmas}
\Crefname{lemma}{Lemma}{Lemmas}
\crefname{proposition}{Proposition}{Propositions}
\Crefname{proposition}{Proposition}{Propositions}
\crefname{corollary}{Corollary}{Corollaries}
\Crefname{corollary}{Corollary}{Corollaries}
\crefname{definition}{Definition}{Definitions}
\Crefname{definition}{Definition}{Definitions}
\crefname{assumption}{Assumption}{Assumptions}
\Crefname{assumption}{Assumption}{Assumptions}
\crefname{example}{Example}{Examples}
\Crefname{example}{Example}{Examples}
\crefname{remark}{Remark}{Remarks}
\Crefname{remark}{Remark}{Remarks}

\newcommand{\R}{\mathbb{R}}
\newcommand{\E}{\mathbb{E}}
\newcommand{\Prob}{\mathbb{P}}
\newcommand{\X}{\mathcal{X}}
\newcommand{\Z}{\mathcal{Z}}
\newcommand{\F}{\mathcal{F}}
\newcommand{\inner}[2]{\langle #1,\, #2\rangle}
\newcommand{\norm}[1]{\lVert #1\rVert}
\newcommand{\abs}[1]{\lvert #1\rvert}
\newcommand{\ind}[1]{\mathbf{1}\{#1\}}
\newcommand{\Ot}{\widetilde{O}}          
\newcommand{\Tht}{\widetilde{\Theta}}    
\DeclareMathOperator*{\argmin}{arg\,min}
\DeclareMathOperator{\esssup}{ess\,sup}
\DeclareMathOperator{\Var}{Var}
\DeclareMathOperator{\Cov}{Cov}
\DeclareMathOperator{\LMO}{LMO}
\newcommand{\dtv}{d_{\mathrm{TV}}}
\newcommand{\dmix}{d_{\mathrm{mix}}}
\newcommand{\tmix}{\tau_{\mathrm{mix}}}

\newcommand{\gap}{\mathcal{G}}           
\newcommand{\Gs}{G_{\sigma}}             
\newcommand{\Ghat}{\widehat{G}}          
\newcommand{\Lam}{\Lambda}               
\newcommand{\sM}{\sigma_{\mathrm{M}}}    
\newcommand{\LM}{L_{\mathrm{M}}}         
\newcommand{\sbar}{\bar{s}}
\newcommand{\dbar}{\bar{\delta}}
\newcommand{\dbarM}{\bar{\delta}_{\mathrm{M}}}
\newcommand{\kapM}{\kappa_{\mathrm{M}}}
\newcommand{\BdotM}{\dot{B}_{\mathrm{M}}}
\newcommand{\BddotM}{\ddot{B}_{\mathrm{M}}}
\newcommand{\Bdot}{\dot{B}}
\newcommand{\Bddot}{\ddot{B}}
\newcommand{\Rcal}{\mathcal{R}}          
\newcommand{\Scal}{\mathcal{S}}          
\newcommand{\Lcal}{\mathcal{L}}          
\newcommand{\mlmc}[1]{\hat{\mu}^{\mathrm{MLMC}}_{#1}}
\newcommand{\ghat}{\hat{g}}
\newcommand{\gpre}{g_t^{\mathrm{pre}}}
\newcommand{\spre}{s_t^{\mathrm{pre}}}
\newcommand{\clip}{\Pi_{\Ghat}}
\newcommand{\etabar}{\bar{\eta}}
\newcommand{\jmax}{j_{\max}}
\newcommand{\algname}{\textsf{MC-ALFCG}}
\newcommand{\basealg}{\textsf{ALFCG}}
\newcommand{\basevar}{\textsf{ALFCG-MVR2}}

\newcommand{\paperfitdisplay}[1]{%
  \ifdefined\JOTAmode
    \resizebox{\linewidth}{!}{$\displaystyle #1$}%
  \else
    #1%
  \fi
}

\newcommand{\stepmark}[1]{\raisebox{-1pt}{\textcircled{\raisebox{-0.6pt}{\scriptsize #1}}}}
\newcommand{\ceq}[1]{\overset{\stepmark{#1}}{=}}
\newcommand{\cle}[1]{\overset{\stepmark{#1}}{\le}}
\newcommand{\cge}[1]{\overset{\stepmark{#1}}{\ge}}

\title{\bf Variance-Reduced Conditional Gradient Methods under\\ Markovian Sampling for Nonconvex Composite Optimization}

\author{
  Zhaojun Peng\\
  School of Computer Science\\
  Nanjing University of Information Science and Technology\\
  \texttt{202592020013@nuist.edu.cn}
}

\hypersetup{
  pdftitle={Variance-Reduced Conditional Gradient Methods under Markovian Sampling for Nonconvex Composite Optimization},
  pdfauthor={Zhaojun Peng, Yuhang Zhang},
  pdfsubject={Stochastic nonconvex optimization under Markovian sampling},
  pdfkeywords={stochastic nonconvex optimization, Markovian sampling,
    variance reduction, conditional gradient method, multilevel Monte Carlo}
}

\begin{document}
\maketitle

\begin{abstract}
We study stochastic composite nonconvex optimization over a compact convex
set when gradient samples arrive along a single trajectory of a fixed
ergodic Markov chain. Existing single-trajectory variance-reduction theory
covers smooth unconstrained objectives. We address the distinct
projection-free composite setting and use the generalized Frank--Wolfe gap
as the stationarity measure. We propose \algname{} (Markov-Chain Adaptive
Lipschitz-Constant-Free Conditional Gradient), which combines a
momentum-based conditional-gradient method with coupled capped multilevel
Monte Carlo estimation and per-iteration clipping. The deepest nested
average uses consecutive states from the same trajectory, giving conditional
bias $O(\tmix/T)$ uniformly over the starting chain state, while coupling
makes the gradient-difference second moment depend on the iterate
displacement. Clipping enforces the pathwise bounds required by the adaptive
analysis. We prove a reduction to the independent-sampling recursion under
$\sigma^2\mapsto2\Lam\Gs^2$ and $L^2\mapsto2\Lam L^2$, where
$\Lam=O(\tmix\log T)$. For positive centered noise, the tuned method attains
expected sample complexity
$\Ot((\tmix^{2}\Gs+\tmix^{5/2}\Gs^2)\varepsilon^{-3}
+\tmix^{5}\varepsilon^{-2})$ to obtain an expected generalized
Frank--Wolfe gap at most $\varepsilon$. An exactly noiseless specialization
uses a separate parameter choice and attains $\Ot(\varepsilon^{-2})$ with
mixing-time-free constants. A mixing-time-oblivious variant, which also
requires no centered noise bound, attains
$\Ot(\tmix^{6}\varepsilon^{-3}+\tmix^{3}\varepsilon^{-2})$.
No variant requires the smoothness constant. The tuned method requires
certified mixing-time and noise bounds, clipping requires a known
$\Ghat\ge G$, and all complexity statements are in expectation under a
fixed transition kernel. For fixed $D>0$ and $\Gs>0$, a parameter-tuning
lower bound matches the leading mixing exponent within the transferred
analysis. Two controlled numerical studies examine dependence sensitivity,
a genuinely nonconvex composite instance, and the clipping mechanism without
claiming absolute empirical superiority.

\end{abstract}
\section{Introduction}\label{sec:intro}

We consider the constrained stochastic composite nonconvex problem
\begin{equation}\label{eq:problem}
  \min_{x\in\X\subset\R^n}\; F(x) \coloneqq f(x) + h(x),
  \qquad f(x) = \E_{z\sim\pi}\big[f(x;z)\big],
\end{equation}
where $\X$ is compact and convex, $h$ is proper, closed, and convex, and
$f$ is differentiable but possibly nonconvex. Samples
$z_1,z_2,\ldots$ are successive states of a single ergodic Markov chain
with stationary distribution $\pi$. The samples are consumed in order, without
restarts or independent resampling. We target settings in which Euclidean
projection onto $\X$ is expensive but the linearized composite subproblem
over $\X$ can be solved efficiently. This geometry motivates the composite
linear minimization oracle introduced in \cref{sec:setup}.

The fixed-chain streaming model arises when a data source cannot be reset
or shuffled, including optimization driven by Markov chain Monte Carlo,
learning from a single time series or sensor stream, dependent-data matrix
factorization, and system identification on a non-resettable system. It also
covers off-policy evaluation under a fixed behavior policy with bounded
importance ratios. In these settings, the stochastic gradient is
conditionally biased given the past, and this bias is correlated across
iterations. MaC-SPIDER shows that recursive variance reduction is possible
for smooth unconstrained objectives on a single Markovian trajectory
\citep{NEURIPS2025_1cdbce34}. The question addressed here is different: can
recursive variance reduction be combined with a projection-free composite
oracle and an adaptive conditional-gradient analysis?

This question must be distinguished from the episodic policy-gradient
model. In the episodic model, trajectories are independent across
iterations and Markovian only within an episode. Methods such as SVRPG,
STORM-PG, and ProxHSPGA exploit this cross-trajectory independence
\citep{papini2018svrpg,yuan2020stormpg,pham2020hybrid}. In the streaming
model studied here, dependence persists across iterations because all
samples belong to one chain. Existing results therefore settle adjacent
oracle models or optimization geometries, but not their projection-free
composite intersection. \Cref{tab:rates} summarizes this boundary.

\begin{table}[t]
\centering
\caption{Representative sample complexities in adjacent nonconvex oracle
and problem classes. The rows are not direct comparisons because the
feasible geometry, regularity assumptions, and stationarity measures
differ. ``Streaming Markov'' denotes a single trajectory of an ergodic
chain. Logarithmic factors are suppressed, and both bounds reported for
this paper are expectation guarantees. The compact oblivious expression
treats the finite centered-noise bound as a fixed problem constant.}
\label{tab:rates}
\vspace{0.4em}
\resizebox{\linewidth}{!}{%
\begin{tabular}{@{}llll@{}}
\toprule
Oracle & Method class & Complexity & Reference \\
\midrule
i.i.d. & SGD (no VR) & $\Ot(\varepsilon^{-4})$ & classical \\
i.i.d. & SPIDER / STORM (VR) & $\Ot(\varepsilon^{-3})$ (optimal) &
  \citep{fang2018spider,cutkosky2019momentum} \\
i.i.d. & \basevar{} (VR, projection-free) & $\Ot(\sigma\varepsilon^{-3}+\varepsilon^{-2})$ &
  \citep{yuan2026alfcg} \\
\midrule
streaming Markov & prox-SGD / AdaGrad / heavy-ball & $\Ot(\varepsilon^{-4})$ &
  \citep{alacaoglu2023convergence} \\
streaming Markov & MaC-SPIDER (VR, smooth and unconstrained) &
  $O(\tmix\pi_{\min}^{-1/2}\varepsilon^{-3})$ &
  \citep{NEURIPS2025_1cdbce34} \\
streaming Markov & projection-free composite CG with recursive VR &
  none known & \\
streaming Markov & \textbf{this paper} (VR, projection-free) &
  $\Ot\big((\tmix^{2}\Gs+\tmix^{5/2}\Gs^2)\varepsilon^{-3}
  +\tmix^{5}\varepsilon^{-2}\big)$ &
  \cref{thm:tuned} \\
streaming Markov & \textbf{this paper} (oblivious variant) &
  $\Ot\big(\tmix^{6}\varepsilon^{-3}+\tmix^{3}\varepsilon^{-2}\big)$ &
  \cref{thm:oblivious} \\
\bottomrule
\end{tabular}
}
\end{table}

Two obstacles prevent a direct replacement of the independent-sampling
oracle. First, a finite Markovian burst produces a state-dependent
conditional bias. Consequently, the gradient-difference estimator no
longer satisfies the conditional-unbiasedness identity used in standard
recursive variance-reduction arguments. The analysis must instead control
this bias within a conditional, contractive error recursion. The
multilevel Monte Carlo estimator of \citep{dorfman2022adapting} controls the
bias of a single gradient, but its conditional bias must still be absorbed
by the recursive estimator used here.

Second, capped multilevel sampling produces rare corrections whose
magnitude grows with the horizon. The resulting estimator has a useful
conditional second moment but no horizon-independent pathwise bound at the
scale required by the adaptive step-size analysis. In fact, its error can
have essential supremum of order $T\Gs$ (\cref{prop:esssup}). Conditional
moment control alone is therefore insufficient for transferring the base
conditional-gradient argument.

We address the two obstacles separately. The proposed method, \algname{}
(\cref{alg:main}), starts from the momentum variance-reduced method
\basevar{} of \citep{yuan2026alfcg}. A coupled multilevel estimator
evaluates consecutive iterates on the same burst of chain samples, so the
difference variance scales with the iterate displacement and the
conditional mean equals a sufficiently deep chain average. The chain thus
mixes within each burst, yielding conditional bias of order $\tmix/T$
uniformly in the initial chain state. We then project the recursive
gradient estimate onto a known ball of radius $\Ghat\ge G$. This clipping
step enforces the pathwise bound required by the adaptive analysis.

Together, the two modifications recover the structure of the
independent-sampling error recursion under the substitutions
\begin{equation*}
  \sigma^2 \;\longmapsto\; \sM^2 = 2\Lam\Gs^2,
  \qquad
  L^2 \;\longmapsto\; \LM^2 = 2\Lam L^2,
  \qquad
  \Lam = O(\tmix\log T).
\end{equation*}
Here $\Gs$ is a certified upper bound on the centered gradient noise. A
formal transfer result verifies the conditional recursion, the pathwise
interface, and the weighted cumulative-error bound required by the
adaptive descent analysis (\cref{thm:reduction}). Explicit propagation is
necessary because the effective noise and recursion smoothness now grow
with $\Lam$ rather than remaining fixed problem constants.

The resulting expected sample complexity for the generalized Frank--Wolfe
gap is
$\Ot((\tmix^{2}\Gs+\tmix^{5/2}\Gs^2)\varepsilon^{-3}
+\tmix^{5}\varepsilon^{-2})$ (\cref{thm:tuned}). The two stochastic
branches cover the low- and high-effective-noise regimes and both vanish
with $\Gs$. The exactly noiseless case uses a separate parameter choice
and attains $\Ot(\varepsilon^{-2})$ with mixing-time-free constants
(\cref{cor:zeronoise}). These conclusions preserve the
Lipschitz-constant-free property of the base method: the smoothness constant
does not enter the algorithm.

The transfer also yields a mixing-time-oblivious variant with expected
sample complexity
$\Ot(\tmix^{6}\varepsilon^{-3}+\tmix^{3}\varepsilon^{-2})$
(\cref{thm:oblivious}). For fixed $D>0$ and $\Gs>0$, an in-analysis lower
bound shows that retuning the free parameters cannot improve the
$\Lam^{5/6}$ exponent of the leading stochastic coefficient within the
transferred constant propagation (\cref{thm:lower}). This is neither an
algorithmic nor an information-theoretic lower bound, and it does not rule
out a sharper clipping-aware analysis. Two controlled numerical studies
further examine dependence sensitivity, a nonconvex composite instance, and
the clipping mechanism.

The tuned parameters require upper bounds on the mixing time and centered
noise level, whereas the oblivious variant removes both requirements at the
cost of worse mixing dependence. Clipping requires a known gradient-norm
bound. All sample-complexity guarantees are in expectation and assume a
fixed transition kernel. Decision-dependent sampling processes are outside
the model. \cref{sec:discussion} gives the precise scope of these
qualifications.

The remainder of the paper is organized as follows. \cref{sec:related}
reviews the closest literature, \cref{sec:setup} states the assumptions and
stationarity measure, and \cref{sec:algorithm} presents \algname{}.
\cref{sec:results} develops the convergence analysis, and
\cref{sec:conclusion} concludes the paper. A notation summary, complete
proofs, and detailed numerical results
\ifdefined\JOTAmode
are provided in Online Resource~1.
\else
are provided in the Appendix.
\fi
\section{Related Work}\label{sec:related}

The closest literature concerns stochastic optimization with Markovian
data, recursive variance reduction, projection-free composite optimization,
episodic policy-gradient methods, and clipping under heavy-tailed noise. We
organize the discussion by oracle model and problem geometry because these
distinctions determine whether an existing guarantee applies to the setting
considered here.

\paragraph{First-order methods with Markovian data.}
Early analyses include ergodic mirror descent \citep{duchi2012ergodic},
Markov-chain gradient descent \citep{sun2018markov}, and non-asymptotic
biased stochastic approximation \citep{karimi2019nonasymptotic}. For
nonconvex constrained problems, \citep{alacaoglu2023convergence} analyzes
projected SGD, AdaGrad-norm, heavy-ball, and proximal SGD with a convex
regularizer, obtaining $\Ot(\varepsilon^{-4})$ rates under a mixing
condition. These methods neither use recursive variance reduction nor call
a conditional-gradient oracle. The multilevel estimator of
\citep{dorfman2022adapting}, on which our estimator is based, gives
mixing-time-oblivious rates for smooth non-composite problems, including an
$\Ot(\varepsilon^{-4})$ nonconvex rate without variance reduction. Related
results include sharper SGD bounds \citep{even2023sgd}, convex
mirror-descent and variational-inequality methods
\citep{solodkin2024methods}, and stochastically constrained convex
optimization \citep{kim2023stochastic}.

Recursive variance reduction under Markovian sampling was posed as an open
direction in \citep{wang2022stability}. MaC-SPIDER subsequently resolved the
smooth unconstrained case using a single trajectory
\citep{NEURIPS2025_1cdbce34}. Under mean-squared smoothness, it attains
$O(\tmix\pi_{\min}^{-1/2}\varepsilon^{-3})$ and provides an
algorithm-independent $\Omega(\tau\varepsilon^{-3})$ lower bound in terms
of a hitting-time scale $\tau$. Our setting differs in its compact
constraint geometry, composite linear oracle, Lipschitz-constant-free
adaptation, and generalized Frank--Wolfe-gap stationarity measure.

\paragraph{Recursive variance reduction under independent sampling.}
SPIDER \citep{fang2018spider} and STORM \citep{cutkosky2019momentum} attain
the optimal $\Ot(\varepsilon^{-3})$ complexity for smooth nonconvex
problems, matching the lower bound in \citep{arjevani2023lower}. STORM+
\citep{levy2021storm} removes parameter knowledge, and composite extensions
include \citep{zhang2020prox}. Standard analyses use conditionally unbiased
gradient differences. A fixed-chain Markovian burst violates this identity,
so the resulting conditional bias must be controlled within the recursive
error analysis.

\paragraph{Conditional-gradient methods.}
The Frank--Wolfe method \citep{frank1956algorithm,jaggi2013revisiting}
avoids projection by calling a linear minimization oracle. Nonconvex
analysis began with \citep{lacoste2016convergence}. Stochastic and
variance-reduced variants include
\citep{reddi2016stochastic,hazan2016variance,mokhtari2020stochastic}. Our
base method is the adaptive Lipschitz-constant-free framework \basealg{} of
\citep{yuan2026alfcg}. Its momentum variant \basevar{} handles convex
composite terms through a linear oracle and attains
$\Ot(\sigma\varepsilon^{-3}+\varepsilon^{-2})$ for the generalized
Frank--Wolfe gap under independent sampling. A related line develops
Lipschitz-constant-free proximal variance-reduced methods
\citep{yuan2025aepg}. We extend the conditional-gradient rather than the
proximal framework because the compact feasible set and linear-oracle
geometry are central to the present Markovian construction.

Two recent general frameworks are especially close in methodology. An
abstract estimator analysis for stochastic Frank--Wolfe methods
\citep{nazykov2024stochastic} does not require unbiased gradient estimates,
but it does not instantiate that condition for a single fixed-chain Markov
trajectory or the composite LMO used here. A Lipschitz-constant-free step
strategy combined with stochastic and variance-reduced Frank--Wolfe
estimators \citep{nandhan2026boosted} likewise does not cover fixed-chain
Markov sampling or the present composite oracle. These results
therefore provide neighboring estimator frameworks rather than the
conditional recursion and pathwise adaptive interface established here.

\paragraph{Episodic variance-reduced policy gradient.}
SVRPG \citep{papini2018svrpg}, STORM-PG \citep{yuan2020stormpg}, and
ProxHSPGA \citep{pham2020hybrid} obtain variance-reduced policy-gradient
rates, with ProxHSPGA also handling composite objectives. Their trajectories
are independent across iterations and are coupled across parameters by
importance sampling. This episodic model differs from the fixed-chain
streaming model considered here, where temporal dependence persists across
iterations.

\paragraph{Clipping and heavy-tailed estimators.}
Gradient clipping is a standard safeguard for heavy-tailed stochastic
gradients \citep{zhang2020why,gorbunov2020clipped,cutkosky2021high}. In our
setting, the multilevel multiplier makes the raw estimator heavy-tailed,
whereas the adaptive descent analysis requires a pathwise error bound.
Clipping at a mixing-time-free radius enforces that interface and does not
increase the pointwise estimation error relative to the true gradient
(\cref{lem:clipsafe}). Multilevel estimation itself originates in
\citep{giles2008multilevel,blanchet2015unbiased}.

Existing results therefore cover either recursive variance reduction for
smooth unconstrained objectives on a single Markovian trajectory or
projection-free composite optimization under independent sampling. To our
knowledge, the remaining intersection addressed here has not been analyzed:
a fixed-chain streaming oracle, a composite linear minimization oracle,
Lipschitz-constant-free adaptation, and generalized Frank--Wolfe-gap
stationarity.
\section{Assumptions and Preliminaries}\label{sec:setup}

We first fix the notation, assumptions, oracle, and stationarity measure
used throughout the analysis.

\paragraph{Notation.}
$\norm{\cdot}$ is the Euclidean norm. For a convex set $B\subset\R^n$,
$\Pi_B$ denotes the Euclidean projection. For $r>0$ we write
$\Pi_r \coloneqq \Pi_{\{g:\norm{g}\le r\}}$.
$\E_{t-1}[\cdot] \coloneqq \E[\cdot\mid\F_{t-1}]$ denotes conditional
expectation with respect to the filtration defined in
\cref{sec:algorithm}. $\Ot(\cdot)$ and $\Tht(\cdot)$ hide polylogarithmic
factors in $T$, $\tmix$, and $1/\varepsilon$, and $\log \coloneqq \log_2$.
A summary of all recurring symbols is provided in \cref{app:notation}.

\paragraph{Assumptions.}
We adopt the following assumptions regarding problem structure, function
regularity, and oracle capabilities.

\begin{assumption}[Problem structure]\label{asm:problem}
$\X\subset\R^n$ is compact and convex with diameter
$D \coloneqq \sup_{x,y\in\X}\norm{x-y}$. The function
$h:\R^n\to\R\cup\{+\infty\}$ is proper, closed, and convex, is finite and
$G$-Lipschitz on $\X$, and $x_0\in\X\cap\operatorname{dom}h$. The feasible domain is
nonempty, the value $F^\ast\coloneqq\inf_{x\in\X}F(x)$ is finite, and it
is attained at some $x^\ast\in\X$. A sufficient condition for the stated
Lipschitz property is $\partial h(x)\ne\varnothing$ and
$\norm{\zeta}\le G$ for every $x\in\X$ and $\zeta\in\partial h(x)$.
\end{assumption}

\begin{assumption}[Individual smoothness]\label{asm:smooth}
For every $z\in\Z$, $f(\cdot\,;z)$ is differentiable on $\X$ and
$\norm{\nabla f(x;z)-\nabla f(y;z)}\le L\norm{x-y}$ for all $x,y\in\X$.
Moreover, differentiation and stationary expectation commute:
$\E_\pi[\nabla f(x;z)]=\nabla f(x)$ for every $x\in\X$.
Consequently $f=\E_\pi[f(\cdot\,;z)]$ is $L$-smooth.
\end{assumption}

\begin{assumption}[Bounded gradients and known clipping radius]\label{asm:bounded}
$\norm{\nabla f(x)}\le G$ for all $x\in\X$, and the algorithm has access to a
constant $\Ghat$ with $\Ghat\ge G$.
\end{assumption}

We use $G$ as a common upper bound for the Lipschitz modulus of $h$ on
$\X$ and for $\norm{\nabla f}$ on $\X$. This is without loss after
replacing separate finite bounds by their maximum. The clipping radius
$\Ghat$ must bound this common constant.

\begin{assumption}[Bounded centered noise]\label{asm:noise}
There exists a finite constant $\Gs$ satisfying
$\sup_{x\in\X,\,z\in\Z}\norm{\nabla f(x;z)-\nabla f(x)}\le\Gs$.
\end{assumption}

\begin{assumption}[Markovian oracle]\label{asm:markov}
$(z_k)_{k\ge0}$ is a time-homogeneous ergodic Markov chain on $\Z$ with
transition kernel $P$ and stationary distribution $\pi$. Its mixing behavior
is measured by
\begin{equation*}
  \dmix(k) \coloneqq \sup_{z\in\Z}\,
  \dtv\big(P^{k}(z,\cdot),\,\pi\big),
  \qquad
  \tmix \coloneqq \min\{k\ge1:\dmix(k)\le\tfrac14\},
\end{equation*}
with $\tmix<\infty$. As usual $\dmix$ is non-increasing and
$\dmix(\ell\,\tmix)\le2^{-\ell}$ for all $\ell\ge1$
\citep{levin2017markov}, whence the summability bound
$\sum_{k\ge1}\dmix(k)\le2\tmix$ that we use repeatedly.
The initial state may be deterministic or random and need not have law
$\pi$.
\end{assumption}

For the tuned method, a certified integer bound
$\widehat\tau_{\mathrm{mix}}\ge\tmix$ may replace $\tmix$ in all tuned
parameters and guarantees. The resulting complexity is then expressed in
$\widehat\tau_{\mathrm{mix}}$. The displayed rates use the canonical
$\tmix$ for readability.

Individual smoothness, rather than smoothness only in expectation, is also
assumed by the base method \basevar{} \citep{yuan2026alfcg}. It is needed
because the recursive estimator evaluates two iterates on the same sample.
Assumption~\ref{asm:noise} instead controls centered gradient noise. The quantity $\Gs$
is distinct from the common deterministic bound $G$.

The tuned method uses a certified value of $\Gs$, while clipping uses
$\Ghat$ from \cref{asm:bounded}. Neither quantity is a smoothness constant.
The analysis does not assume that the chain starts in stationarity or that
the iterate sequence is ergodic. The oblivious variant of
\cref{thm:oblivious} requires knowledge of neither $\tmix$ nor $\Gs$.

\paragraph{Linear minimization oracle (LMO).}
Following \citep{yuan2026alfcg}, we assume access to an oracle that, for
any $g\in\R^n$ (typically a gradient estimate), returns a point in the
solution set
\begin{equation*}
  \LMO(g) \coloneqq \argmin_{v\in\X}\;\inner{v}{g}+h(v) .
\end{equation*}

\paragraph{Stationarity measure.}
Since $F(\cdot)$ is possibly nonconvex, we use the generalized
Frank--Wolfe gap as the stationarity measure.

\begin{definition}[FW gap]\label{def:fwgap}
For any $x\in\X$, the generalized Frank--Wolfe gap is defined as
\begin{equation}\label{eq:fwgap}
  \gap(x) \coloneqq \inner{\nabla f(x)}{x-\tilde v}
            + h(x)-h(\tilde v),
  \qquad \tilde v\in\LMO(\nabla f(x)),
\end{equation}
i.e., the largest decrease of the linearized objective over $\X$.
\end{definition}

The gap is nonnegative, and $\gap(x)=0$ if and only if $x$ is a
first-order stationary point of \cref{eq:problem} in the directional sense.
On a compact domain it is the natural projection-free analogue of the
gradient mapping. When the population gradient is available, one call to
the LMO evaluates the gap. Under the stochastic oracle it serves as the
population stationarity measure used by the analysis. A point $x\in\X$
with $\gap(x)\le\varepsilon$ is called an $\varepsilon$-approximate
stationary point.
\section{The Proposed Algorithm}\label{sec:algorithm}

The Markov-Chain Adaptive Lipschitz-Constant-Free Conditional Gradient
(\algname{}) method modifies \basevar{} \citep{yuan2026alfcg} in two
respects. It replaces each independent mini-batch by a coupled capped
multilevel burst drawn from the current Markov-chain state, and it clips the
resulting recursive gradient estimate to a known ball. All adaptive
quantities are computed before the burst and are therefore measurable with
respect to the past filtration. The remaining conditional-gradient update
is inherited from the base method. \cref{sec:mlmc,sec:clipping} describe the
new components, and \cref{sec:adaptive} summarizes the inherited adaptive
quantities.

\begin{algorithm}[t]
\DontPrintSemicolon
\caption{\algname{}: Markov-chain adaptive Lipschitz-constant-free
conditional gradient with momentum variance reduction}
\label{alg:main}
\KwIn{horizon $T$, parameters $\rho>0$ and $\beta>0$, clipping radius
$\Ghat\ge G$, initial point $x_0\in\X$, and chain state $z_0$.}
\textbf{Initialize:} $x_{-1}=x_0$, $g_{-1}=0$, $\alpha_0=1$,
$\jmax=\lfloor\log T\rfloor$\;
\For{$t=0,1,\dots,T$}{
  \tcp{S1) adaptive constants (all $\F_{t-1}$-measurable)}
  \lIf{$t=0$}{$\hat\alpha_t=\alpha_t=1$}
  \lElse{$u_i \coloneqq \beta+L_i^2\norm{x_{i+1}-x_i}^2$,\quad
  $\hat\alpha_t=\Big(\frac{1+\max_{i<t}u_i}{1+\sum_{i<t}u_i}\Big)^{2/3}$,\quad
  $\alpha_t=\min_{0\le k\le t}\hat\alpha_k$}
  $L_t=\rho\big(1+\sum_{i<t}L_i^2\norm{x_{i+1}-x_i}^2\big)^{1/2}\alpha_t^{-1/4}$\;
  \tcp{S2) coupled MLMC burst (\cref{eq:mlmc})}
  draw $J_t\sim\mathrm{Geom}(\tfrac12)$ independently of $\F_{t-1}$ and the
  subsequent transition randomness and advance the chain by $N_t$ steps\;
  for $y\in\{x_t,x_{t-1}\}$ compute
  $\ghat_t(y)\coloneqq\mlmc{t}[\nabla f(y;\cdot)]$\;
  \tcp{S3) momentum variance reduction (STORM-type, no anchor) + clipping}
  $\gpre=(1-\alpha_t)\big(g_{t-1}-\ghat_t(x_{t-1})\big)+\ghat_t(x_t)$,\quad
  $g_t=\clip(\gpre)$\;
  \tcp{S4) composite LMO and adaptive short step}
  $v_t\in\LMO(g_t)$\;
  \lIf{$v_t=x_t$}{$\etabar_t=0$}
  \lElse{$\etabar_t=\min\Big(
     \frac{h(x_t)-h(v_t)-\inner{g_t}{v_t-x_t}}{L_t\norm{v_t-x_t}^2},\,1\Big)$}
  $x_{t+1}=x_t+\etabar_t\,(v_t-x_t)$\;
}
\KwOut{$x_{\hat t}$ with $\hat t$ drawn independently and uniformly from
$\{0,\dots,T\}$}
\end{algorithm}

\subsection{Gradient Estimation Mechanism: Coupled Capped MLMC}
\label{sec:mlmc}

Fix a horizon $T$ and let $\jmax \coloneqq \lfloor\log T\rfloor$, so that
$2^{\jmax}\in[T/2,\,T]$. At iteration $t$ the algorithm draws
$J_t\sim\mathrm{Geom}(\tfrac12)$ (i.e.\ $\Prob[J_t=j]=2^{-j}$, $j\ge1$),
independently of $\F_{t-1}$ and of the subsequent chain-transition
randomness, and then advances the chain by
\begin{equation*}
  N_t \coloneqq
  \begin{cases}
    2^{J_t}, & J_t\le\jmax,\\
    1, & J_t>\jmax,
  \end{cases}
\end{equation*}
steps, obtaining consecutive states $z_t^{(1)},\dots,z_t^{(N_t)}$ (the chain
continues from wherever the previous iteration left it). For an
$\F_{t-1}$-measurable map $\varphi:\Z\to\R^n$ define the prefix averages and
the estimator
\begin{equation}\label{eq:mlmc}
\begin{aligned}
  \hat\mu^{j}[\varphi]
  &\coloneqq \frac{1}{2^{j}}\sum_{i=1}^{2^{j}}
  \varphi\big(z_t^{(i)}\big),\\
  \mlmc{t}[\varphi]
  &\coloneqq \hat\mu^{0}[\varphi] +
  \begin{cases}
    2^{J_t}\big(\hat\mu^{J_t}[\varphi]-\hat\mu^{J_t-1}[\varphi]\big),
      & J_t\le\jmax,\\
    0, & J_t>\jmax .
  \end{cases}
\end{aligned}
\end{equation}
Two implementation details are essential. First, $J_t$ is drawn
\emph{before} any samples: if $J_t>\jmax$ only the single base-level sample
is drawn. Without this cap the expected burst length would be
$\sum_j 2^{-j}\,2^{j}=\infty$. With it,
$\E[N_t]\le 1+\log_2 T$ (\cref{lem:mlmc}(d)).
Second, the momentum update evaluates the oracle at both $x_t$ and
$x_{t-1}$ \emph{on the same samples and the same $J_t$}. By linearity of
\cref{eq:mlmc},
\begin{equation*}
  \widehat{\Delta}_t \coloneqq \ghat_t(x_t)-\ghat_t(x_{t-1})
  = \mlmc{t}\big[\nabla f(x_t;\cdot)-\nabla f(x_{t-1};\cdot)\big],
\end{equation*}
so the gradient-\emph{difference} estimate is itself an MLMC estimator of a
function whose centered magnitude is at most $2L\norm{x_t-x_{t-1}}$ by
individual smoothness (\cref{asm:smooth}). The coupling does not increase
the number of chain transitions because \basevar{} already evaluates two
iterates on each sampled batch. It makes the variance of the difference
proportional to the squared displacement.

Throughout, sample complexity counts consumed Markov-chain states
(equivalently, chain transitions up to the initial state). Each consumed
state supports the two coupled gradient evaluations at $x_t$ and
$x_{t-1}$, so the number of stochastic gradient evaluations is at most
twice the reported sample count. This constant factor does not alter any
displayed rate.

\paragraph{Filtration.}
$\F_{t-1}$ collects everything generated before the $t$-th burst: all
iterates, all previous bursts and $J_s$ ($s<t$), and the current chain state.
In particular, $\F_{-1}$ contains the initial point and initial chain state.
it need not be trivial.
Consequently, $x_t$, $x_{t-1}$, $g_{t-1}$, $\alpha_t$, and $L_t$ are
$\F_{t-1}$-measurable. The ordering in \cref{alg:main} ensures this
property. The burst at iteration $t$ starts from the
$\F_{t-1}$-measurable chain state.

\subsection{The Clipping Safeguard}\label{sec:clipping}

The analysis of the adaptive stepsize machinery consumes a bound on the
error $s_t\coloneqq g_t-\nabla f(x_t)$ \emph{along each sample path} (see
\cref{sec:results}). The clipping step supplies the pathwise interface
required by the transferred analysis. For the unclipped estimator no such
bound exists at any useful scale, as the following proposition shows.

\begin{proposition}[Heavy tail of the raw estimator, proof in
\cref{app:pathwise}]\label{prop:esssup}
There is a two-state chain and an oracle satisfying
\cref{asm:smooth,asm:bounded,asm:noise,asm:markov} such that the unclipped
initial error $s_0^{\mathrm{pre}}=\ghat_0(x_0)-\nabla f(x_0)$ obeys
$\esssup\norm{s_0^{\mathrm{pre}}}\ge(T/2-1)\,\Gs$ for every $T\ge2$.
\end{proposition}

\noindent The event $J_0=\jmax$ combined with a half-and-half sign pattern
realizes the $2^{J}$ multiplier. Clipping at radius $\Ghat$ removes the
issue by construction: $\norm{g_t}\le\Ghat$, hence
$\norm{s_t}\le G+\Ghat\le2\Ghat$ and
$\delta_t\coloneqq L_t\norm{x_{t+1}-x_t}\le G+\Ghat$ almost surely
(\cref{lem:clipsafe}). Because the true gradient lies inside the clipping
ball, projection is non-expansive toward it,
$\norm{s_t}\le\norm{\spre}$ pointwise. Thus clipping does not increase the
pointwise estimation error relative to the true gradient, and every
conditional-moment upper bound remains valid. The radius $\Ghat$ is a fixed
problem constant that does \emph{not} depend on $\tmix$ or $T$. Thus, the
safeguard preserves mixing-time obliviousness. The corner convention
$v_t=x_t\Rightarrow\etabar_t\coloneqq0$ makes the short-step update
well-defined when its displayed quotient has zero denominator.

\subsection{Adaptive Lipschitz Estimation and the Short Step}
\label{sec:adaptive}

The adaptive quantities $\alpha_t$ and $L_t$ are inherited from
\basevar{}. The self-normalized weight $\alpha_t$ controls the momentum
recursion, whereas $L_t$ tracks curvature through past iterate
displacements without requiring prior knowledge of the global smoothness
constant. After the composite LMO returns $v_t$, the short-step rule
minimizes a quadratic upper model along the conditional-gradient direction.
Compactness of $\X$ then gives $\norm{x_{t+1}-x_t}\le D$. Moreover,
$\alpha_t\ge(t+1)^{-2/3}$, and the expected burst length is at most
$1+\log_2 T$. Hence expected sample complexity differs from iteration
complexity by only a logarithmic factor (proofs in
\cref{app:reduction}).

\paragraph{Parameter regimes.}
We analyze two positive-noise instantiations of \cref{alg:main}. The
\emph{tuned} regime sets $\rho=\sqrt{\Lam}$ and
$\beta=2\Lam\Gs^2$, where
$\Lam\coloneqq306\,\tmix(1+\log_2 T)$. It requires upper bounds on $\tmix$
and $\Gs$ and, for fixed $D>0$ and $\Gs>0$, attains the $\Lam^{5/6}$ exponent
of the leading stochastic coefficient. Within the transferred analysis,
that exponent cannot be improved by retuning $(\rho,\beta)$
(\cref{thm:tuned,thm:lower}). The \emph{oblivious} regime uses fixed
$\rho,\beta=\Theta(1)$ and requires neither bound
(\cref{thm:oblivious}). The exactly noiseless specialization uses
$\rho=\Theta(1)$ and $\beta=(T+1)^{-1}$ (\cref{cor:zeronoise}).
\section{Convergence and Sample Complexity}\label{sec:results}

This section provides the convergence analysis of \algname{}. Throughout,
\cref{asm:problem,asm:smooth,asm:bounded,asm:noise,asm:markov} are in
force, \algname{} is run for $T+1$ iterations, and we write
\begin{equation}\label{eq:effective}
\begin{aligned}
  \Lam &\coloneqq 306\,\tmix\,(1+\log_2 T),
  & \sM^2 &\coloneqq 2\Lam\Gs^2,\\
  \LM^2 &\coloneqq 2\Lam L^2,
  & T_0 &\coloneqq \big\lceil(128\,\tmix)^{3/4}\big\rceil .
\end{aligned}
\end{equation}
These quantities are, respectively, the second-moment inflation factor and
the effective noise, effective smoothness, and burn-in scales. We further use the
error and step magnitudes
\begin{equation*}
  s_t \coloneqq g_t-\nabla f(x_t),
  \qquad
  \delta_t \coloneqq L_t\norm{x_{t+1}-x_t} .
\end{equation*}

The analysis has four steps. We first establish conditional-bias and
second-moment bounds for the coupled estimator. We then use clipping to
obtain the pathwise bounds required by the adaptive step-size argument.
These ingredients yield a conditional error recursion with the same
structure as the independent-sampling recursion after two explicit
coefficient substitutions. Finally, we propagate the resulting constants
to obtain the tuned and oblivious complexities and characterize the best
leading mixing exponent obtainable by parameter tuning within this
analysis.

\subsection{Conditional Moments of the Coupled MLMC Estimator}
\label{sec:res-mlmc}

We begin with the structural properties of the estimator
\cref{eq:mlmc}. All moment bounds below are conditional on $\F_{t-1}$
and hold uniformly in the chain state from which the burst starts. No
ergodicity of the iterate process is invoked anywhere.

\begin{lemma}[MLMC conditional moments, proof in \cref{app:mlmc}]
\label{lem:mlmc}
Let $\varphi:\Z\to\R^n$ be $\F_{t-1}$-measurable with
$\bar\varphi\coloneqq\E_\pi[\varphi]$ and
$\widetilde H\coloneqq\sup_z\norm{\varphi(z)-\bar\varphi}$.
Then, for every realization of the $\F_{t-1}$-measurable chain state from
which the burst starts,
\begin{enumerate}[label=(\alph*),leftmargin=2.2em,itemsep=1pt,topsep=2pt]
\item $\E\big[\mlmc{t}[\varphi]\,\big|\,\F_{t-1}\big]
       =\E\big[\hat\mu^{\jmax}[\varphi]\,\big|\,\F_{t-1}\big]$
      \hfill (conditional telescoping).
\item $\big\lVert\E\big[\mlmc{t}[\varphi]\,\big|\,\F_{t-1}\big]-\bar\varphi\big\rVert
       \le 4\widetilde H\tmix/2^{\jmax}\le 8\widetilde H\tmix/T$.
\item $\E\big[\norm{\mlmc{t}[\varphi]-\bar\varphi}^2\,\big|\,\F_{t-1}\big]
       \le \Lam\,\widetilde H^2$.
\item $\E[N_t]\le 1+\log_2 T$.
\end{enumerate}
\end{lemma}

The telescoping identity in part (a) is the key estimator property. Since
$J_t$ is independent of the chain and the
levels are nested prefix averages of one burst, conditioning on
$\F_{t-1}$ (which is \emph{pre}-$J_t$) telescopes the level corrections:
\begin{equation*}
\paperfitdisplay{\begin{aligned}
\E\big[\mlmc{t}[\varphi]\,\big|\,\F_{t-1}\big]
={}& \E\big[\hat\mu^{0}\,\big|\,\F_{t-1}\big]
   +\sum_{j=1}^{\jmax}2^{-j}\cdot2^{j}
    \Big(\E\big[\hat\mu^{j}\,\big|\,\F_{t-1}\big]
        -\E\big[\hat\mu^{j-1}\,\big|\,\F_{t-1}\big]\Big)\\
={}& \E\big[\hat\mu^{\jmax}\,\big|\,\F_{t-1}\big].
\end{aligned}}
\end{equation*}
The conditional mean is thus the average over $2^{\jmax}\ge T/2$
\emph{consecutive} samples, and the vector-valued total-variation bound
$\norm{\int\varphi\,d(\mu-\nu)}\le2\,\dtv(\mu,\nu)\,\widetilde H$ (no
dimension factor) gives, uniformly in the start state,
\begin{equation*}
\paperfitdisplay{\begin{aligned}
\Big\lVert\E\big[\mlmc{t}[\varphi]\,\big|\,\F_{t-1}\big]-\bar\varphi\Big\rVert
\;\le\;& \frac{1}{2^{\jmax}}\sum_{i=1}^{2^{\jmax}}
   2\widetilde H\,\dmix(i)
\;\le\; \frac{4\widetilde H\,\tmix}{2^{\jmax}}
\;\le\; \frac{8\widetilde H\,\tmix}{T},
\end{aligned}}
\end{equation*}
which is part (b). Bias control is thereby decoupled from the iterate
trajectory: the classical small-step argument (bounding
$\norm{x_t-x_{t-\tau}}$) is never used. For the second moment (c), the
per-level variance must be capped at the trivial bound for levels
$2^j\lesssim\tmix$. Below the mixing scale the correlation bound is
vacuous. This cap is exactly what limits the inflation to
$\Lam=O(\tmix\log T)$ instead of $O(\tmix^2)$.

We apply \cref{lem:mlmc} to the two functions estimated by the algorithm.
Define $n_t\coloneqq\ghat_t(x_t)-\nabla f(x_t)$ and
$e_t\coloneqq\widehat\Delta_t-(\nabla f(x_t)-\nabla f(x_{t-1}))$.

\begin{corollary}[Two instantiations of one lemma, proof in
\cref{app:reduction}]\label{cor:inst}
Conditionally on $\F_{t-1}$,
\begin{center}
\ifdefined\JOTAmode
\resizebox{\linewidth}{!}{%
\fi
\begin{tabular}{@{}lccc@{}}
\toprule
error & $\widetilde H\le$ & conditional bias $\le$ & conditional 2nd moment $\le$\\
\midrule
$n_t$ & $\Gs$ & $8\Gs\tmix/T$ & $\Lam\Gs^2$\\
$e_t$ & $\min\{2L\norm{x_t-x_{t-1}},\,2\Gs\}$ &
  $16L\tmix\norm{x_t-x_{t-1}}/T$ &
  $4\Lam L^2\norm{x_t-x_{t-1}}^2$\\
\bottomrule
\end{tabular}
\ifdefined\JOTAmode
}%
\fi
\end{center}
\end{corollary}

The bound on $e_t$ uses individual smoothness (\cref{asm:smooth}): the
coupled difference is an MLMC estimate of
$z\mapsto\nabla f(x_t;z)-\nabla f(x_{t-1};z)$, whose centered magnitude is
at most $2L\norm{x_t-x_{t-1}}$ \emph{per sample}. Note also that both rows
use the \emph{centered} bound: every mixing-time factor multiplies a power
of $\Gs$ (or of the step), which yields the noise-sensitive coefficients
below.

\subsection{Pathwise Safety by Construction}\label{sec:res-pathwise}

Next, we turn to the pathwise bounds. The downstream analysis of
\citep{yuan2026alfcg} is not an expectation-only argument: the
stepsize-ratio bound and the logarithmic truncations apply the
displacement bound $\delta_t\le\dbar$ along each sample path. The
displacement obeys the pathwise inequality (both stepsize branches,
\cref{app:pathwise})
\begin{equation*}
\begin{aligned}
\delta_t
\;\le\;& G+\norm{g_t}
\;\le\; 2G+\norm{s_t},
\end{aligned}
\end{equation*}
so a pathwise bound on $\delta_t$ is, in substance, a pathwise bound on
$\norm{s_t}$. For the raw MLMC estimator, no such bound exists at any
useful scale (\cref{prop:esssup}). The clipping step supplies exactly
this bound by construction, and does so without touching any conditional
moment the analysis uses.

\begin{lemma}[Clipping enforces the pathwise bounds, proof in
\cref{app:pathwise}]\label{lem:clipsafe}
For every $t\ge0$, almost surely:
\begin{enumerate}[label=(\alph*),leftmargin=2.2em,itemsep=1pt,topsep=2pt]
\item $\norm{g_t}\le\Ghat$ and
      $\norm{s_t}\le G+\Ghat\le2\Ghat \eqqcolon \sbar$.
\item $\delta_t\le G+\Ghat\le2\Ghat$.
\item $\norm{s_t}\le\norm{\spre}$ pointwise, where
      $\spre\coloneqq\gpre-\nabla f(x_t)$.
\end{enumerate}
\end{lemma}

Part (c) uses the non-expansiveness of projection toward a ball containing
the true gradient. This property lets the error recursion below run on the
clipped sequence with no additional bias term. Part (b) holds in both
branches of the stepsize rule and under the corner convention. It is the
pathwise displacement bound that the adaptive analysis invokes. The transfer
uses two structural facts: the descent inequality of
\citep[Lemma~6]{yuan2026alfcg} is pathwise algebra valid for an
\emph{arbitrary} $g_t$, and the momentum identity enters at the error
recursion replaced by \cref{thm:reduction}. For the raw estimator, the
corresponding quantities can have essential supremum $\Theta(T\Gs)$ by
\cref{prop:esssup}.

\subsection{The Reduction Theorem}\label{sec:res-reduction}

The conditional moments and pathwise bounds yield an error recursion with
the same form as its independent-sampling counterpart. Markov dependence
enters this recursion through two explicit scalar substitutions.

\begin{theorem}[Reduction to the i.i.d.\ recursion, proof in
\cref{app:reduction}]\label{thm:reduction}
Let $T\ge T_0$. Then
$\E_{-1}\norm{s_0}^2\le\Lam\Gs^2$ almost surely, and for every
$1\le t\le T$,
\begin{equation}\label{eq:reduction}
  \E_{t-1}\norm{s_t}^2
  \;\le\;
  (1-\alpha_t)\norm{s_{t-1}}^2
  \;+\;2\alpha_t^2\,\sM^2
  \;+\;8\,\LM^2\,\norm{x_t-x_{t-1}}^2 .
\end{equation}
\end{theorem}

Inequality \cref{eq:reduction} has the same form as the i.i.d.\
error recursion of \citep[Lemma~26(a)]{yuan2026alfcg}, whose right-hand side
is $(1-\alpha_t)\norm{s_{t-1}}^2+2\alpha_t^2\sigma^2
+8L^2\norm{x_t-x_{t-1}}^2$. The required substitutions are
\begin{equation}\label{eq:substitution}
  \sigma^2\;\longmapsto\;\sM^2=2\Lam\Gs^2,
  \qquad
  L^2\;\longmapsto\;\LM^2=2\Lam L^2
  \quad\text{(inside the recursion only)}.
\end{equation}
To show how the conditional bias is absorbed, decompose the momentum error as
$\spre=(1-\alpha_t)s_{t-1}+(1-\alpha_t)e_t+\alpha_t n_t$, and clipping
only shrinks it: $\norm{s_t}\le\norm{\spre}$ pointwise
(\cref{lem:clipsafe}(c)). Writing
$\xi_t\coloneqq(1-\alpha_t)e_t+\alpha_t n_t$ and
$b_t\coloneqq\E_{t-1}[\xi_t]$,
\begin{equation*}
\begin{aligned}
\E_{t-1}\norm{s_t}^2
\;\le\;& \norm{(1-\alpha_t)s_{t-1}}^2
        +2(1-\alpha_t)\inner{s_{t-1}}{b_t}
        +\E_{t-1}\norm{\xi_t}^2\\
\cle{1}\;& (1-\alpha_t)\norm{s_{t-1}}^2
        +\tfrac{2}{\alpha_t}\norm{b_t}^2
        +\E_{t-1}\norm{\xi_t}^2 ,
\end{aligned}
\end{equation*}
where \stepmark{1} is Young's inequality with weight $\alpha_t/2$ on the
cross term, using
$(1-\alpha_t)\big(1-\tfrac{\alpha_t}{2}\big)\le1-\alpha_t$.
No independence between the burst and $s_{t-1}$ is invoked: the
cross-time correlation is paid for by the $2/\alpha_t$ factor and
absorbed by the contraction $(1-\alpha_t)$. By \cref{cor:inst} and the
stepsize lower bound $\alpha_t\ge(t+1)^{-2/3}$, each component of the
bias cost $\tfrac2{\alpha_t}\norm{b_t}^2$ is at most
$128\,\tmix\,T^{-4/3}\le1$ times its variance sibling once $T\ge T_0$.
doubling the variance constants absorbs it, which is the factor $2$ in
$\sM^2$ and $\LM^2$. This contraction-based absorption is not available
unchanged in an anchored SPIDER recursion, whose within-epoch error is
additive rather than damped by $(1-\alpha_t)$. MaC-SPIDER controls that
different structure through Markov-aware batching
\citep{NEURIPS2025_1cdbce34}. Here the momentum variant is the compatible
base for the adaptive conditional-gradient chain.

Because \cref{eq:reduction} is conditional, has the same
$\F_{t-1}$-measurable coefficients as its i.i.d.\ counterpart, and is
accompanied by the pathwise bounds of \cref{lem:clipsafe}, the downstream
analysis of \citep[Appendix~H]{yuan2026alfcg} can be transferred after the
substitutions \cref{eq:substitution} and $\sbar\mapsto2\Ghat$ are verified.
\cref{app:pathwise,app:rates} carries out this transfer term by term.
\cref{tab:watershed} records which constants acquire a factor of $\Lam$.

\begin{table}[ht]
\centering
\caption{Which constants inflate under the substitution
\cref{eq:substitution}. ``Source of $L$'' distinguishes the smoothness
constant of the descent inequality (never inflated) from the recursion
coefficient (inflated).}
\label{tab:watershed}
\vspace{0.4em}
\begin{tabular}{@{}llll@{}}
\toprule
constant & definition & source of its $L$ or $\sigma$ & inflates?\\
\midrule
$\dbar$ & $8(LD+C+\sbar+\rho D)$ & descent inequality & no ($\sbar=2\Ghat$: no)\\
$\kappa$ & $(2+2\dbar^2/\rho)^{1/2}$ & via $\dbar$ & no\\
$\Bdot$ & $18(1+\beta+\dbar^2)$ & via $\dbar$ & no (but carries $\beta$)\\
$\Bddot$ & $40L^2\kappa^2\rho^{-1}\ln(1+\dbar^2T)$ &
  \textbf{recursion coefficient} & \textbf{yes: $L^2\mapsto\LM^2$}\\
$Y_1,\;Z_1$ & $\propto\rho^2$, $\propto\rho^6$ & neither & no (only via $\rho$)\\
$\sigma^2$ (in $\Omega$) & oracle variance & \textbf{recursion} &
  \textbf{yes: $\mapsto\sM^2$}\\
\bottomrule
\end{tabular}
\end{table}

\subsection{Complexity of the Tuned Variant}\label{sec:res-tuned}

\begin{theorem}[Tuned complexity, proof in \cref{app:rates}]
\label{thm:tuned}
Fix a problem-class envelope $\overline G_\sigma>0$ and suppose
$0<\Gs\le\overline G_\sigma$.
Run \algname{} with
$\rho=\rho^\star\coloneqq\sqrt{\Lam}$ and
$\beta=\beta^\star\coloneqq\sM^2=2\Lam\Gs^2$, and suppose $T\ge T_0$.
Then, with $\Lcal\coloneqq\max\{1,\ln(1+\dbarM^2T)\}=O(\log(\Lam T))$,
\begin{equation}\label{eq:tunedbound}
  \E\Big[\tfrac1{T+1}\textstyle\sum_{t=0}^{T}\gap(x_t)\Big]
  \;\le\;
  \Rcal\,(T+1)^{-1/3}+\Scal\,(T+1)^{-1/2},
\end{equation}
where
\begin{equation*}
  \Rcal\;\le\;C_{\Rcal}\,\Lcal
  \big(\Gs^{1/3}\Lam^{2/3}+\Gs^{2/3}\Lam^{5/6}\big),
  \qquad
  \Scal\;\le\;C_{\Scal}\,\Lcal^{3/2}\,\Lam^{5/2},
\end{equation*}
with $C_{\Rcal},C_{\Scal}$ explicit constants depending only on
$(L,D,G,\Ghat,\overline G_\sigma)$, given in \cref{app:rates}. In
particular, these constants are uniform over
$0<\Gs\le\overline G_\sigma$.
Consequently, for any $0<\varepsilon\le1$, the expected number of samples
needed to reach $\E[\gap(x_{\hat t})]\le\varepsilon$ is at most
\begin{equation}\label{eq:tunedcomplexity}
  \Ot\big(\tmix^{2}\,\Gs\,\varepsilon^{-3}
        \;+\;\tmix^{5/2}\,\Gs^{2}\,\varepsilon^{-3}
        \;+\;\tmix^{5}\,\varepsilon^{-2}\big).
\end{equation}
\end{theorem}

We sketch where the exponent comes from. Treating problem constants as
$O(1)$ and $\beta=\Theta(\sM^2)$, the constant map in
\cref{tab:watershed} gives
\begin{equation*}
  \frac{Y_0}{\rho}
  \;=\;\Theta\Big(
   \underbrace{\tfrac{\Lam}{\rho}}_{\Bdot}
   +\underbrace{\rho}_{Y_1}
   +\underbrace{\tfrac{\Lam}{\rho^2}+\tfrac{\Lam}{\rho}}_{\Bddot}
  \Big)
  \qquad\Longrightarrow\qquad
  \rho^\star=\Theta(\sqrt{\Lam}),
  \quad
  \frac{Y_0}{\rho^\star}=\Theta(\sqrt{\Lam}).
\end{equation*}
The remaining factor is exact:
$\beta^{1/3}/\vartheta=\beta^{1/6}+\beta^{1/3}$. Hence the two stochastic
branches are $\Tht(\Gs^{1/3}\Lam^{2/3})$ and
$\Tht(\Gs^{2/3}\Lam^{5/6})$.  When $\beta^\star\ge1$, the latter is the
larger branch and yields the exponent $5/2=3\times\tfrac56$. When
$\beta^\star<1$, the former gives the continuous low-noise companion.
The free parameter $\rho$ is what removes one factor of $\sqrt\Lam$ from
$\Bddot$.

\begin{remark}\label{rem:loworder}
We highlight some features of \cref{thm:tuned}.
\stepmark{1} \textbf{The lower-order term.}
The $\varepsilon^{-2}$ term of \cref{eq:tunedcomplexity} is not free
of the mixing time: at $\rho^\star=\sqrt{\Lam}$ the burn-in constant
contains a $\rho^6$ contribution. When $D>0$, the upper bound
$\Scal=\Ot(\Lam^{5/2})$ is matched by a lower bound of the same order within
the transferred constant propagation (\cref{app:rates}). For $D=0$, only the
upper bound is asserted.
\stepmark{2} \textbf{Two noise regimes.} If
$\beta^\star=2\Lam\Gs^2\ge1$, the
$\tmix^{5/2}\Gs^2\varepsilon^{-3}$ branch dominates its
$\tmix^2\Gs\varepsilon^{-3}$ companion.  If $\beta^\star<1$, the latter
must be retained. Dropping it would make the displayed bound
discontinuous at small positive noise.
\stepmark{3} \textbf{Scope of the tuning certificate.} For fixed $D>0$
and $\Gs>0$, \cref{thm:lower} shows that no re-tuning of $(\rho,\beta)$ within
the transferred constant propagation can improve the exponent $\Lam^{5/6}$
of the leading
rate coefficient.  It does not assert uniform optimality in $\Gs$ or an
exact minimizer of the full coefficient.
\end{remark}

\begin{corollary}[Noise-vanishing and exactly noiseless regimes]\label{cor:zeronoise}
With $\overline G_\sigma$ fixed, both stochastic $\varepsilon^{-3}$ terms of
\cref{eq:tunedcomplexity} vanish as $\Gs\to0$. In the exactly noiseless
case $\Gs=0$, choose $\rho=\Theta(1)$ and
$\beta=(T+1)^{-1}$. Then $s_t\equiv0$ and the exact-error specialization
of the transferred constant propagation gives
\begin{equation*}
  \E\Big[\tfrac1{T+1}\textstyle\sum_{t=0}^{T}\gap(x_t)\Big]
  =\Ot((T+1)^{-1/2}),
\end{equation*}
so the expected sample complexity is
$\Ot(\varepsilon^{-2})$ with mixing-time-free constants.
\end{corollary}

The positive-noise and exactly noiseless conclusions use different
parameter branches. The latter is not obtained by substituting
$\Gs=0$ into $\beta^\star$.  The centered noise level $\Gs$ and the
noise-independent clipping radius make this separation explicit.

\subsection{Complexity of the Mixing-Time-Oblivious Variant}
\label{sec:res-oblivious}

The tuned parameters $(\rho^\star,\beta^\star)$ require (upper bounds
on) $\tmix$ and $\Gs$. We next show that neither is necessary for the
$\varepsilon^{-3}$ rate itself.

\begin{theorem}[Oblivious complexity, proof in \cref{app:oblivious}]
\label{thm:oblivious}
Run \algname{} with any fixed $\rho,\beta=\Theta(1)$ chosen independently of
$\tmix$ and $\Gs$ (the clipping radius $\Ghat$ is likewise $\tmix$-free), and let
$T\ge T_0$. Then, for any $0<\varepsilon\le1$, the expected number of samples
needed to reach $\E[\gap(x_{\hat t})]\le\varepsilon$ is at most
\begin{equation*}
  \Ot\big(\tmix^{6}\,\varepsilon^{-3}+\tmix^{3}\,\varepsilon^{-2}\big).
\end{equation*}
This shorthand treats the finite noise bound $\Gs$ as a fixed problem
constant. More explicitly, the proof gives
\begin{equation*}
  \Ot\Big(
    (\tmix^3+\tmix^6\Gs^6)\varepsilon^{-3}
    +(\tmix^3+\tmix^2\Gs^4)\varepsilon^{-2}
  \Big).
\end{equation*}
Thus the method is oblivious to the \emph{value} of $\Gs$, but the guarantee
is not uniform over an unbounded family of noise levels.
\end{theorem}

The algorithm of \cref{thm:oblivious} uses no knowledge of the mixing time
and retains an $\varepsilon^{-3}$ stochastic term for the generalized
Frank--Wolfe gap. The comparison with the tuned result is regime dependent:
the tuned high-effective-noise branch carries
$\tmix^{5/2}\Gs^2\varepsilon^{-3}$, whereas its low-effective-noise
companion carries $\tmix^2\Gs\varepsilon^{-3}$. The oblivious bound carries
$\tmix^6\varepsilon^{-3}$ after fixed noise constants are suppressed.
Results in \citep{alacaoglu2023convergence,dorfman2022adapting} concern
different problem structures and stationarity measures, so none of these
expressions is interpreted as a direct rate dominance.

\subsection{Optimality of the Exponent within the Analysis}
\label{sec:res-lower}

We next characterize the best exponent obtainable by tuning within the
transferred analysis.

\begin{theorem}[In-analysis lower bound, proof in \cref{app:lower}]
\label{thm:lower}
Consider the rate coefficient $\Rcal(\rho,\beta)$ produced by the transferred
constant propagation from \citep[Appendix~H]{yuan2026alfcg} under the
substitutions \cref{eq:substitution}. Fix $D>0$ and $\Gs>0$, and let any
admissible pathwise bound $\dbar\ge0$ be used in the propagation. Then, for all $\rho>0$
and $\beta>0$,
\begin{equation*}
  \Rcal(\rho,\beta)\;\ge\;288\sqrt2\;D\,\sM^{5/3}
  \;=\;288\sqrt2\cdot2^{5/6}\,D\,\Gs^{5/3}\,\Lam^{5/6}.
\end{equation*}
For fixed $(D,\Gs)$, the exponent $\Lam^{5/6}$ in the tuned upper bound
is therefore optimal within the transferred constant propagation, up to
logarithms. This
does not claim uniform optimality in $\Gs$ or that
$(\rho^\star,\beta^\star)$ exactly minimizes the full coefficient.
\end{theorem}

The proof uses two inequalities in the transferred constant propagation
that are independent of $\dbar$:
$Y_0\ge2Y_1\ge24D^2\rho^2$ and $Y_0\ge96\Bdot\ge1728\beta$. Therefore
\begin{equation*}
\paperfitdisplay{\begin{aligned}
\Rcal
\;\ge\;& 2\,\frac{Y_0}{\rho}\,
        \max\big(\beta^{1/3},\,\sM^2\beta^{-2/3}\big)
\;\ge\; 2\max\Big(24D^2\rho,\tfrac{1728\beta}{\rho}\Big)
        \max\big(\beta^{1/3},\sM^2\beta^{-2/3}\big)\\
\;\ge\;& 288\sqrt2\,D\;\sqrt{\beta}\,
        \max\big(\beta^{1/3},\sM^2\beta^{-2/3}\big),
\end{aligned}}
\end{equation*}
using $\max(a,b)\ge\sqrt{ab}$, and the map
$b\mapsto\sqrt b\max(b^{1/3},\sM^2b^{-2/3})$ is minimized exactly at
$b=\sM^2$ with value $\sM^{5/3}$. The floor $\Omega(\Lam^{5/6})$ follows.

This certificate concerns parameter tuning within the transferred analysis.
it is not an algorithmic or information-theoretic lower bound. The next
proposition complements it by showing that the pre-clipping estimator can
attain the multilevel second-moment scale, without making a claim about the
clipped estimator.

\begin{proposition}[Pre-clipping second-moment inflation, proof in
\cref{app:lower}]\label{prop:inflation}
There is a two-state chain and an oracle satisfying
\cref{asm:smooth,asm:bounded,asm:noise,asm:markov} such that the MLMC
gradient error at $t=0$ obeys
\begin{equation*}
  \E\norm{s_0^{\mathrm{pre}}}^2
  \;\ge\;\frac{\Gs^2\,\tmix}{16}\,
  \log\!\Big(\frac{T}{64\,\tmix}\Big)-\Gs^2 ,
\end{equation*}
which is $\Omega(\tmix\log T\cdot\Gs^2)=\Omega(\Lam\Gs^2)$ up to absolute
constants when $T$ is sufficiently large relative to $\tmix$ and the
displayed right-hand side is positive. This construction certifies only the
asymptotic pre-clipping moment scale. It is not a lower bound for the
clipped estimator and is not used by the upper-complexity theorem.
\end{proposition}

\subsection{Scope and Limitations}
\label{sec:discussion}

\begin{remark}\label{rem:claims}
Three limitations delimit the results. First, clipping is an algorithmic
modification that enforces the pathwise error bound required by the
transferred analysis. For the original i.i.d.\ MVR2 estimator of
\citep{yuan2026alfcg}, establishing this bound remains the open challenge
stated in Remark~23. For our raw capped multilevel estimator,
\cref{prop:esssup} rules out a horizon-independent bound at the required
scale. Clipping therefore sidesteps, rather than resolves, the original
unclipped question.

Second, the capped multilevel estimator has expected logarithmic sample
cost but remains heavy-tailed, so all sample-complexity statements are in
expectation. Third, \cref{thm:lower} concerns parameter tuning within the
transferred constant propagation. It is neither an information-theoretic
lower bound nor a lower bound for the clipped estimator. A different
estimator or a clipping-aware Lyapunov analysis may therefore improve the
mixing dependence.
\end{remark}
\section{Conclusions}\label{sec:conclusion}

In this paper, we introduced \algname{} for nonconvex composite optimization
under a fixed-chain streaming Markovian oracle. MaC-SPIDER already provides
single-trajectory variance reduction for smooth unconstrained optimization
\citep{NEURIPS2025_1cdbce34}. To our knowledge, \algname{} gives the first
guarantee at the different intersection of recursive variance reduction, a
projection-free composite LMO, and Lipschitz-constant-free conditional-gradient
steps. The technical core is a reduction: a coupled MLMC estimator controls
the state-dependent conditional bias, clipping supplies the pathwise
interface, and the resulting momentum recursion transfers to the adaptive
i.i.d.\ descent analysis. The resulting two scalar substitutions,
together with explicit constant propagation, yield
$\Ot((\tmix^{2}\Gs+\tmix^{5/2}\Gs^2)\varepsilon^{-3}
+\tmix^{5}\varepsilon^{-2})$ expected sample complexity for the generalized
Frank--Wolfe gap. An exactly noiseless specialization separately attains
$\Ot(\varepsilon^{-2})$ with mixing-time-free constants. For fixed $D>0$
and positive noise, a parameter-tuning lower bound matches the leading
mixing-inflation exponent within the transferred analysis. We also obtain a
mixing-time-oblivious variant at
$\Ot(\tmix^{6}\varepsilon^{-3}+\tmix^{3}\varepsilon^{-2})$ for fixed
problem constants. The hidden constant is not uniform in the centered-noise
bound. The method remains oblivious to that bound's value and preserves
Lipschitz-constant-freeness. The tuned stochastic terms remain
noise-sensitive and vanish with $\Gs$. The tuned variant requires certified
bounds on both $\tmix$ and $\Gs$. The oblivious variant requires neither.
Two controlled numerical illustrations indicate reduced relative
deterioration under stronger dependence and show that clipping suppresses
rare estimator excursions on a nonconvex composite instance. They do not
show absolute empirical superiority.

\paragraph{Limitations.}
Our scope is the fixed-chain streaming model. The theory does not cover
on-policy reinforcement learning, where the transition kernel depends on the
decision variable. It also excludes generally nonstationary replay buffers and
temporal-difference learning, whose semi-gradient need not be the gradient of
an objective of the form $f+h$. It does cover the fixed-chain applications
listed in \cref{sec:intro}, including off-policy evaluation under a fixed
behavior policy. Within the model, several costs remain. Clipping requires a
known radius $\Ghat\ge G$. The tuned variant requires mixing-time knowledge,
although the oblivious variant does not. The complexity statements hold in
expectation, and the base method requires convex $h$.

\paragraph{Outlook.}
Three questions strike us as natural next steps. First,
\emph{clipping-aware variance analysis}: our bounds estimate the second
moment of the pre-clipped error. Whether clipping provably reduces the
effective $\sM^2$ itself is open, and would attack the
$\tmix$-exponent from the side that \cref{thm:lower} leaves open. Second,
\emph{information-theoretic lower bounds} in $\tmix$ for the composite
conditional-gradient/FW-gap setting. MaC-SPIDER already gives such lower
bounds for smooth unconstrained Markovian variance reduction
\citep{NEURIPS2025_1cdbce34}, but they do not directly cover our geometry.
Third, the
same reduction template plausibly applies to other gradient-based methods
whose analyses are driven by a STORM-type recursion, including
proximal-linearized ADMM and orthogonalized-momentum methods. The mechanism
(mixing inside the burst, bias absorbed by contraction) is not specific to
conditional gradient. Carrying it out requires reconciling each method's
additional schedules with the Markovian bias term and is left to future
work.

\bibliographystyle{plainnat}
\bibliography{refs}

\clearpage
\appendix
\section{Notation Summary}\label[appendix]{app:notation}

\cref{tab:notation-problem,tab:notation-analysis} collect the recurring
symbols of the paper, grouped by origin. Symbols are listed in order of
first appearance.

\begin{table}[ht]
\centering
\caption{Notation: problem, oracle, and algorithm.}
\label{tab:notation-problem}
\small
\begin{tabular}{@{}ll@{}}
\toprule
symbol & meaning \\
\midrule
\multicolumn{2}{@{}l}{\emph{Problem and assumptions (\cref{sec:setup})}}\\
$\X$, $D$ & feasible set (compact and convex) and its diameter \\
$f(\cdot\,;z)$, $f$, $h$, $F$ & component loss, $f=\E_{z\sim\pi}[f(\cdot\,;z)]$, convex part, and $F=f+h$ \\
$x^\ast$ & a minimizer of $F$ over $\X$ \\
$L$ & individual smoothness constant (\cref{asm:smooth}) \\
$G$ & bound on $\norm{\nabla f(x)}$ and on subgradients of $h$ \\
$\Ghat$ & known clipping radius, $\Ghat\ge G$ (\cref{asm:bounded}) \\
$\Gs$ & centered noise level $\sup_{x,z}\norm{\nabla f(x;z)-\nabla f(x)}$ \\
$P$, $\pi$ & transition kernel and stationary distribution (\cref{asm:markov}) \\
$\dmix(k)$, $\tmix$ & mixing coefficient and mixing time \\
$\LMO(g)$ & composite linear minimization oracle $\argmin_{v\in\X}\inner{v}{g}+h(v)$ \\
$\gap(x)$ & generalized Frank--Wolfe gap (\cref{def:fwgap}) \\
$\Pi_B$, $\clip$ & Euclidean projection onto $B$ and clipping operator $\Pi_{\{g:\norm{g}\le\Ghat\}}$ \\
\midrule
\multicolumn{2}{@{}l}{\emph{Algorithm (\cref{alg:main})}}\\
$T$, $\hat t$ & horizon and output index, uniform on $\{0,\dots,T\}$ \\
$\rho$, $\beta$ & parameters of the adaptive scales $L_t$, $\alpha_t$ \\
$u_i$, $\hat\alpha_k$, $\alpha_t$ & momentum-weight accumulators and momentum weight \\
$L_t$ & adaptive Lipschitz-constant-free scale \\
$g_t$, $\gpre$ & clipped gradient estimate and its pre-clipping value \\
$v_t$, $\etabar_t$ & LMO output and adaptive short step \\
\midrule
\multicolumn{2}{@{}l}{\emph{MLMC estimator (\cref{sec:mlmc})}}\\
$J_t$, $N_t$, $\jmax$ & MLMC level $\sim\mathrm{Geom}(\frac12)$, burst length, and level cap $\lfloor\log_2 T\rfloor$ \\
$z_t^{(1)},\dots,z_t^{(N_t)}$ & consecutive chain samples of the $t$-th burst \\
$\hat\mu^{j}[\varphi]$, $\mlmc{t}[\varphi]$ & prefix averages and capped MLMC estimator \cref{eq:mlmc} \\
$\ghat_t(y)$, $\widehat\Delta_t$ & MLMC gradient estimate at $y$ and coupled difference $\ghat_t(x_t)-\ghat_t(x_{t-1})$ \\
$\F_{t-1}$, $\E_{t-1}$ & filtration before the $t$-th burst and $\E[\,\cdot\mid\F_{t-1}]$ \\
\bottomrule
\end{tabular}
\end{table}

\begin{table}[ht]
\centering
\caption{Notation: analysis quantities.}
\label{tab:notation-analysis}
\small
\ifdefined\JOTAmode
\resizebox{\linewidth}{!}{%
\fi
\begin{tabular}{@{}ll@{}}
\toprule
symbol & meaning \\
\midrule
$s_t$, $\spre$ & gradient-estimate error $g_t-\nabla f(x_t)$ and its pre-clipping value \\
$n_t$, $e_t$ & single-point MLMC error and coupled-difference error (\cref{cor:inst}) \\
$\delta_t$ & step magnitude $L_t\norm{x_{t+1}-x_t}$ \\
$\Lam$ & second-moment inflation factor $306\,\tmix(1+\log_2 T)$ \cref{eq:effective} \\
$\sM^2$, $\LM^2$ & effective noise $2\Lam\Gs^2$ and effective smoothness $2\Lam L^2$ \\
$T_0$ & burn-in horizon $\lceil(128\,\tmix)^{3/4}\rceil$ \\
$\sbar$ & pathwise error bound $2\Ghat$ (\cref{lem:clipsafe}) \\
$\dbar$, $\dbarM$ & pathwise displacement bound of the base analysis and its Markov instantiation \\
$\kappa$, $\kapM$ & stepsize-ratio constant of \citep{yuan2026alfcg} and its Markov instantiation \\
$\Bdot$, $\Bddot$ ($\BdotM$, $\BddotM$) & aggregate recursion constants (Markov instantiations) \\
$Y_0,Y_1,Z_0,Z_1$ & constants of the chain of \citep[Appendix~H]{yuan2026alfcg} \\
$\Rcal$, $\Scal$ & coefficients of $(T+1)^{-1/3}$ and $(T+1)^{-1/2}$ in \cref{eq:tunedbound} \\
$\Lcal$ & logarithmic factor $\max\{1,\ln(1+\dbarM^2T)\}$ \\
$\Ot(\cdot)$, $\Tht(\cdot)$ & hide polylogarithmic factors in $T$, $\tmix$, and $1/\varepsilon$, with $\log\coloneqq\log_2$ \\
\bottomrule
\end{tabular}
\ifdefined\JOTAmode
}%
\fi
\end{table}
\section{Proof for \texorpdfstring{\cref{sec:res-mlmc}}{the MLMC Section}: The Conditional-Moment Lemma}\label[appendix]{app:mlmc}

\paragraph{Setup and notation for this appendix.}
Fix $t$ and abbreviate $\varphi_i\coloneqq\varphi(z_t^{(i)})$,
$\E_{t-1}[\cdot]=\E[\cdot\mid\F_{t-1}]$. Conditionally on $\F_{t-1}$ the
burst is a Markov chain started from the $\F_{t-1}$-measurable current
state, i.e.\ $z_t^{(i)}\sim P^{i}(z_{\mathrm{cur}},\cdot)$ given
$\F_{t-1}$. We use the coupling convention of \cref{sec:mlmc}: let
$(z_t^{(i)})_{i\ge1}$ denote the chain continued indefinitely, on which the
prefix averages $\hat\mu^{j}=\hat\mu^{j}[\varphi]$ are defined for all
$j\le\jmax$ simultaneously. The algorithm reveals only the first $N_t$
states, and for the realized $J_t$ the estimator \cref{eq:mlmc} coincides
with the corresponding expression on the coupled trajectory. $J_t$ is
independent of $\F_{t-1}$ and of the trajectory. Finally, for $i<k$ let
$\mathcal{G}_i$ denote the $\sigma$-algebra generated by $\F_{t-1}$ and
$z_t^{(1)},\dots,z_t^{(i)}$, with $\mathcal G_0\coloneqq\F_{t-1}$.

We first record the vector-valued total-variation bound used throughout.

\begin{lemma}[TV--Bochner bound]\label{lem:tvbochner}
Let $\mu,\nu$ be probability measures on $\Z$ and $\psi:\Z\to\R^n$ with
$\sup_z\norm{\psi(z)}\le H$. Then
$\norm{\int\psi\,d(\mu-\nu)}\le 2\,\dtv(\mu,\nu)\,H$.
\end{lemma}

\begin{proof}
With $\lambda\coloneqq\mu-\nu$ (a signed measure of total variation
$\norm{\lambda}_{\mathrm{TV}}=2\dtv(\mu,\nu)$),
\begin{equation*}
\begin{aligned}
\Big\lVert\int\psi\,d\lambda\Big\rVert
\;\le\;& \int\norm{\psi}\,d\abs{\lambda}
\;\le\; H\,\norm{\lambda}_{\mathrm{TV}}
\;=\; 2\,\dtv(\mu,\nu)\,H .
\end{aligned}
\end{equation*}
No dimension factor appears.
\end{proof}

\begin{lemma}[Second moment of a consecutive-segment average]
\label{lem:segment}
For every offset $\ell\ge0$ and length $m\ge1$, let
$\hat\mu_{\ell}^{(m)}\coloneqq
\frac1m\sum_{r=1}^m\varphi_{\ell+r}$ be the average of the $m$ samples
immediately following $\mathcal G_\ell$. Then
\begin{equation*}
  \E\big[\norm{\hat\mu_{\ell}^{(m)}-\bar\varphi}^2
    \,\big|\,\mathcal G_\ell\big]
  \;\le\;\frac{24\,\widetilde H^2\,\tmix}{m}.
\end{equation*}
\end{lemma}

\begin{proof}
\begin{equation*}
  \psi_r\coloneqq\varphi_{\ell+r}.
\end{equation*}
\begin{equation*}
\paperfitdisplay{\begin{aligned}
\E\big[\norm{\hat\mu_{\ell}^{(m)}-\bar\varphi}^2
  \,\big|\,\mathcal G_\ell\big]
={}& \frac1{m^2}\sum_{i=1}^m\sum_{k=1}^m
  \E\big[\inner{\psi_i-\bar\varphi}{\psi_k-\bar\varphi}
    \,\big|\,\mathcal G_\ell\big]\\
\le{}& \frac1{m^2}\Big[\sum_{i=1}^m
  \E\big[\norm{\psi_i-\bar\varphi}^2\,\big|\,\mathcal G_\ell\big]
  +2\sum_{i<k}\E\big[\inner{\psi_i-\bar\varphi}
        {\E[\psi_k-\bar\varphi\mid\mathcal{G}_{\ell+i}]}
    \,\big|\,\mathcal G_\ell\big]\Big]\\
\cle{1}{}& \frac1{m^2}\Big[m\widetilde H^2
  +2\sum_{i<k}\widetilde H\cdot2\widetilde H\,\dmix(k-i)\Big]\\
\cle{2}{}& \frac1{m^2}\Big[m\widetilde H^2+4\widetilde H^2\cdot m\cdot2\tmix\Big]
\;\le\;\frac{9\,\widetilde H^2\tmix}{m}
\;\le\;\frac{24\,\widetilde H^2\tmix}{m},
\end{aligned}}
\end{equation*}
where
\stepmark{1} uses $\norm{\psi_i-\bar\varphi}\le\widetilde H$ pointwise,
the Cauchy--Schwarz inequality, and
$\E[\psi_k-\bar\varphi\mid\mathcal{G}_{\ell+i}]
=\int(\varphi-\bar\varphi)\,
d\big(P^{k-i}(z_t^{(\ell+i)},\cdot)-\pi\big)$
combined with \cref{lem:tvbochner} and the definition of $\dmix$ as a
supremum over the start state.
\stepmark{2} uses
$\sum_{i<k}\dmix(k-i)\le m\sum_{d\ge1}\dmix(d)\le2m\tmix$
(\cref{asm:markov}), and the last two steps use $\tmix\ge1$.
(We state the conservative constant $24$ for headroom in what follows.)
\end{proof}

\begin{lemma}[Per-level variance]\label{lem:perlevel}
For $1\le j\le\jmax$ let
$V_j\coloneqq\E_{t-1}\norm{\hat\mu^{j}-\hat\mu^{j-1}}^2$. Then
\begin{equation*}
  V_j\;\le\;\min\Big\{4\widetilde H^2,\;
  \frac{144\,\widetilde H^2\tmix}{2^{j}}\Big\}.
\end{equation*}
\end{lemma}

\begin{proof}
Write $n\coloneqq2^{j-1}$, and let $\hat A\coloneqq\hat\mu^{j-1}$ (average
of the first $n$ samples) and
$\hat B\coloneqq\frac1n\sum_{i=n+1}^{2n}\varphi_i$ (average of the second
$n$), so that $\hat\mu^{j}-\hat\mu^{j-1}=\tfrac12(\hat B-\hat A)$.
For the trivial cap,
\begin{equation*}
\begin{aligned}
\norm{\hat\mu^{j}-\hat\mu^{j-1}}
\;=\;&\tfrac12\norm{\hat B-\hat A}
\;\le\;\tfrac12\big(\norm{\hat B-\bar\varphi}+\norm{\hat A-\bar\varphi}\big)
\;\le\;\widetilde H
\qquad\text{pointwise},
\end{aligned}
\end{equation*}
so $V_j\le\widetilde H^2\le4\widetilde H^2$. For the mixing bound,
\begin{equation*}
\begin{aligned}
V_j
\;=\;&\tfrac14\,\E_{t-1}\norm{\hat B-\hat A}^2
\;\le\;\tfrac12\,\E_{t-1}\norm{\hat A-\bar\varphi}^2
   +\tfrac12\,\E_{t-1}\norm{\hat B-\bar\varphi}^2\\
\cle{1}\;& \tfrac12\cdot\frac{24\widetilde H^2\tmix}{n}
   +\tfrac12\cdot\frac{24\widetilde H^2\tmix}{n}
\;=\;\frac{24\,\widetilde H^2\tmix}{2^{j-1}}
\;=\;\frac{48\,\widetilde H^2\tmix}{2^{j}}
\;\le\;\frac{144\,\widetilde H^2\tmix}{2^{j}},
\end{aligned}
\end{equation*}
where \stepmark{1} applies \cref{lem:segment} to $\hat A$ with
$\ell=0$, and to $\hat B$ with $\ell=n$ (the segment following the
first-half information), each with $m=n=2^{j-1}$. The tower property
then removes the conditioning on $\mathcal G_n$ in the second bound.
\end{proof}

\begin{proof}[Proof of \cref{lem:mlmc}]
\emph{(a) Conditional telescoping.}
Since $J_t$ is independent of $\F_{t-1}$ and of the coupled trajectory, and
each $\hat\mu^{j}$ is a function of the trajectory alone,
\begin{equation*}
\paperfitdisplay{\begin{aligned}
\E\big[\mlmc{t}[\varphi]\,\big|\,\F_{t-1}\big]
\ceq{1}{}& \E\big[\hat\mu^{0}\,\big|\,\F_{t-1}\big]
+\sum_{j=1}^{\jmax}\Prob[J_t=j]\cdot2^{j}
 \Big(\E\big[\hat\mu^{j}\,\big|\,\F_{t-1}\big]
     -\E\big[\hat\mu^{j-1}\,\big|\,\F_{t-1}\big]\Big)\\
={}& \E\big[\hat\mu^{0}\,\big|\,\F_{t-1}\big]
+\sum_{j=1}^{\jmax}
 \Big(\E\big[\hat\mu^{j}\,\big|\,\F_{t-1}\big]
     -\E\big[\hat\mu^{j-1}\,\big|\,\F_{t-1}\big]\Big)
\;=\;\E\big[\hat\mu^{\jmax}\,\big|\,\F_{t-1}\big],
\end{aligned}}
\end{equation*}
where \stepmark{1} uses
$\E[X\ind{J_t=j}\mid\F_{t-1}]=2^{-j}\,\E[X\mid\F_{t-1}]$ for
trajectory-measurable $X$. Note the filtration is \emph{pre}-$J_t$: the
telescoping is conditional precisely because $\F_{t-1}$ contains no
information about $J_t$.

\emph{(b) Bias.} With $z_t^{(i)}\sim P^{i}(z_{\mathrm{cur}},\cdot)$ given
$\F_{t-1}$,
\begin{equation*}
\paperfitdisplay{\begin{aligned}
\Big\lVert\E\big[\mlmc{t}[\varphi]\,\big|\,\F_{t-1}\big]-\bar\varphi\Big\rVert
={}& \Big\lVert\frac{1}{2^{\jmax}}\sum_{i=1}^{2^{\jmax}}
  \Big(\E\big[\varphi_i\,\big|\,\F_{t-1}\big]-\bar\varphi\Big)\Big\rVert\\
\cle{1}{}& \frac{1}{2^{\jmax}}\sum_{i=1}^{2^{\jmax}}
  2\widetilde H\,\dmix(i)
\;\le\;\frac{2\widetilde H}{2^{\jmax}}\cdot2\tmix
\;=\;\frac{4\widetilde H\tmix}{2^{\jmax}}
\;\le\;\frac{8\widetilde H\tmix}{T},
\end{aligned}}
\end{equation*}
where \stepmark{1} is \cref{lem:tvbochner} applied to
$P^{i}(z_{\mathrm{cur}},\cdot)$ versus $\pi$, and the final step uses
$2^{\jmax}\ge T/2$. The bound is uniform over the start state
$z_{\mathrm{cur}}$. No stationarity or ergodicity of the iterate process
is used.

\emph{(c) Second moment.} Decompose
$\mlmc{t}[\varphi]-\bar\varphi=(\hat\mu^{0}-\bar\varphi)
+2^{J_t}(\hat\mu^{J_t}-\hat\mu^{J_t-1})\ind{J_t\le\jmax}$. Then
\begin{equation*}
\begin{aligned}
\E_{t-1}\norm{\mlmc{t}[\varphi]-\bar\varphi}^2
\;\le\;& 2\,\E_{t-1}\norm{\hat\mu^{0}-\bar\varphi}^2
+2\,\E_{t-1}\Big\lVert2^{J_t}(\hat\mu^{J_t}-\hat\mu^{J_t-1})
   \ind{J_t\le\jmax}\Big\rVert^2\\
\ceq{1}\;& 2\,\E_{t-1}\norm{\hat\mu^{0}-\bar\varphi}^2
+2\sum_{j=1}^{\jmax}2^{-j}\cdot4^{j}\,V_j
\;\le\;2\widetilde H^2+2\sum_{j=1}^{\jmax}2^{j}V_j ,
\end{aligned}
\end{equation*}
where \stepmark{1} again uses $J_t\perp(\F_{t-1},\text{trajectory})$.
only one level is ever drawn, so no cross-level terms appear beyond the
single $(a+b)^2\le2a^2+2b^2$ above.
If $T\ge\tmix$, split the level sum at the mixing scale
(\cref{lem:perlevel}):
\begin{equation*}
\begin{aligned}
\sum_{j=1}^{\jmax}2^{j}V_j
\;\le\;& \sum_{j:\,2^{j}\le\tmix}2^{j}\cdot4\widetilde H^2
  +\sum_{j:\,\tmix<2^{j}\le2^{\jmax}}144\,\widetilde H^2\tmix\\
\cle{1}\;& 8\widetilde H^2\tmix
  +144\,\widetilde H^2\tmix\big(1+\log_2(T/\tmix)\big),
\end{aligned}
\end{equation*}
where \stepmark{1} uses the geometric sum
$\sum_{2^{j}\le\tmix}2^{j}\le2\tmix$ and the level count
$\#\{j:\tmix<2^{j}\le2^{\jmax}\}
\le\lfloor\log_2 T\rfloor-\lfloor\log_2\tmix\rfloor
\le1+\log_2(T/\tmix)$
(the ``$+1$'' from the floors is not optional). Hence
\begin{equation*}
\begin{aligned}
\E_{t-1}\norm{\mlmc{t}[\varphi]-\bar\varphi}^2
\;\le\;& 2\widetilde H^2
  +\big[304+288\log_2(T/\tmix)\big]\widetilde H^2\tmix\\
\;\le\;& 306\,\widetilde H^2\tmix\,(1+\log_2 T)
\;=\;\Lam\widetilde H^2,
\end{aligned}
\end{equation*}
using $2\le2\tmix$, $304+2=306$, $288\le306$, and
$\log_2(T/\tmix)\le\log_2 T$.
If instead $T<\tmix$, every level takes the trivial cap:
$\sum_j2^{j}V_j\le4\widetilde H^2\sum_{j\le\jmax}2^{j}
\le8\widetilde H^2\cdot2^{\jmax}\le8\widetilde H^2T<8\widetilde H^2\tmix$,
so the total is at most
$2\widetilde H^2+16\widetilde H^2\tmix\le18\widetilde H^2\tmix
\le\Lam\widetilde H^2$.
Omitting the trivial cap on the low levels would inflate the sum to
$\Theta(\tmix^2)$ because the chain has not yet mixed and the correlation
bound of \cref{lem:segment} is vacuous there. The cap limits the inflation
to $\Lam=O(\tmix\log T)$.

\emph{(d) Expected samples.} By the capped-sampling rule,
\begin{equation*}
\begin{aligned}
\E[N_t]
\;=\;& \sum_{j=1}^{\jmax}2^{-j}\cdot2^{j}
   +\sum_{j>\jmax}2^{-j}\cdot1
\;=\;\jmax+2^{-\jmax}
\;\le\;\log_2 T+1 .
\end{aligned}
\end{equation*}
Had the algorithm drawn $2^{J_t}$ samples unconditionally, the expectation
would be $\sum_{j\ge1}2^{-j}2^{j}=\infty$. Drawing $J_t$ first and falling
back to the base level when $J_t>\jmax$ is what makes the expected cost
logarithmic.
\end{proof}
\section{Proofs for \texorpdfstring{\cref{sec:res-mlmc,sec:res-reduction}}{the Reduction Sections}}\label[appendix]{app:reduction}

We first record two elementary facts about the adaptive weights.

\begin{lemma}[Deterministic stepsize floor]\label{lem:alphafloor}
For all $t\ge0$, $\alpha_t\ge(t+1)^{-2/3}$, deterministically and with no
upper bound on the $u_i$ required.
\end{lemma}

\begin{proof}
Fix $1\le k\le t$ and write $M_k\coloneqq\max_{i<k}u_i$,
$S_k\coloneqq\sum_{i<k}u_i$, so that $S_k\le kM_k$. Then
\begin{equation*}
\begin{aligned}
\hat\alpha_k
\;=\;&\Big(\frac{1+M_k}{1+S_k}\Big)^{2/3}
\;\ge\;\Big(\frac{1+M_k}{1+kM_k}\Big)^{2/3}
\;\cge{1}\;k^{-2/3},
\end{aligned}
\end{equation*}
where \stepmark{1} holds since
$k(1+M_k)\ge1+kM_k\iff k\ge1$. Together with $\hat\alpha_0=1$,
$\alpha_t=\min_{0\le k\le t}\hat\alpha_k\ge t^{-2/3}\ge(t+1)^{-2/3}$
for $t\ge1$, and $\alpha_0=1$.
\end{proof}

\begin{lemma}[Reciprocal-increment bound]\label{lem:alphaincrement}
For all $t\ge0$,
\begin{equation*}
  0\;\le\;\frac1{\alpha_{t+1}}-\frac1{\alpha_t}\;\le\;\frac23
  \qquad\text{deterministically}.
\end{equation*}
\end{lemma}

\begin{proof}
The case $t=0$ is immediate because $\alpha_0=\alpha_1=1$.
For $k\ge1$, let $S_k\coloneqq\sum_{i<k}u_i$,
$M_k\coloneqq\max_{i<k}u_i$, and
$r_k\coloneqq(1+S_k)/(1+M_k)$, so
$\hat\alpha_k^{-1}=r_k^{2/3}$. Adding $u_k$ gives
$r_{k+1}-r_k\le1$ whenever $r_{k+1}>r_k$. Indeed, if $u_k\le M_k$,
the difference is $u_k/(1+M_k)\le1$. If $u_k>M_k$, then
\begin{equation*}
\begin{aligned}
r_{k+1}-r_k
\;=\;&\frac{S_k}{1+u_k}-\frac{S_k-M_k}{1+M_k}
\;\le\;\frac{M_k}{1+M_k}
\;\le\;1 .
\end{aligned}
\end{equation*}
Since $r_k\ge1$ and $x\mapsto x^{2/3}$ has derivative at most $2/3$ on
$[1,\infty)$,
$(r_{k+1}^{2/3}-r_k^{2/3})_+\le2/3$. Finally,
$\alpha_t^{-1}=\max_{0\le k\le t}\hat\alpha_k^{-1}$, so taking running
maxima preserves the same one-step increment bound.
\end{proof}

\begin{lemma}[Conditional Abel transfer]\label{lem:conditionalabel}
Let $\F_{-1}\subseteq\F_0\subseteq\cdots\subseteq\F_T$ be a filtration.
Let $S_{-1}=0$, let $S_t\ge0$ be
$\F_t$-measurable, let $B_t\ge0$ be $\F_{t-1}$-measurable for
$0\le t\le T$, and let $\alpha_t\in(0,1]$ be $\F_{t-1}$-measurable for
$0\le t\le T+1$. Assume that the displayed expectations below are finite. Set
$A_t\coloneqq\alpha_t^{-1}$. Suppose $\alpha_0=1$,
\begin{equation*}
  \E_{t-1}[S_t]\le(1-\alpha_t)S_{t-1}+B_t,
  \qquad
  0\le\alpha_{t+1}^{-1}-\alpha_t^{-1}\le\frac23 .
\end{equation*}
Then
\begin{equation}\label{eq:conditionalabel}
  \E\Big[\sum_{t=0}^T S_t\Big]
  \;\le\;
  3\,\E\Big[\sum_{t=0}^T\frac{B_t}{\alpha_t}\Big].
\end{equation}
\end{lemma}

\begin{proof}
\begin{equation*}
\begin{aligned}
\E_{t-1}[S_t]
\;\le\;&(1-\alpha_t)S_{t-1}+B_t\\
\Longleftrightarrow{}&
\E_{t-1}[S_t]
\le(A_t-1)\big(S_{t-1}-\E_{t-1}[S_t]\big)+A_tB_t\\
\overset{\stepmark{1}}{\Longrightarrow}{}&
\E[S_t]
\le\E\big[(A_t-1)(S_{t-1}-S_t)+A_tB_t\big].
\end{aligned}
\end{equation*}
\begin{equation*}
\begin{aligned}
\E\Big[\sum_{t=0}^T S_t\Big]
\;\le\;&\E\Big[
  \sum_{t=0}^T(A_t-1)(S_{t-1}-S_t)
  +\sum_{t=0}^TA_tB_t\Big]\\
\ceq{2}{}&\E\Big[
  (A_0-1)S_{-1}+(1-A_{T+1})S_T
  +\sum_{t=0}^T(A_{t+1}-A_t)S_t
  +\sum_{t=0}^TA_tB_t\Big]\\
\le{}&\frac23\E\Big[\sum_{t=0}^T S_t\Big]
  +\E\Big[\sum_{t=0}^T\frac{B_t}{\alpha_t}\Big]\\
\Longrightarrow\quad
\E\Big[\sum_{t=0}^T S_t\Big]
\;\le\;&3\E\Big[\sum_{t=0}^T\frac{B_t}{\alpha_t}\Big].
\end{aligned}
\end{equation*}
where \stepmark{1} uses the $\F_{t-1}$-measurability of $A_t$ and $B_t$
and the tower property. Step \stepmark{2} is Abel summation, with $A_0=1$, $S_{-1}=0$,
$1-A_{T+1}\le0$, and $0\le A_{t+1}-A_t\le2/3$.
\end{proof}

\subsection{Proof of \texorpdfstring{\cref{cor:inst}}{the instantiation corollary}}

\begin{proof}[Proof of \cref{cor:inst}]
Both $x_t$ and $x_{t-1}$ are $\F_{t-1}$-measurable, so
$\varphi^{(1)}(z)\coloneqq\nabla f(x_t;z)$ and
$\varphi^{(2)}(z)\coloneqq\nabla f(x_t;z)-\nabla f(x_{t-1};z)$ are
$\F_{t-1}$-measurable maps. By linearity of \cref{eq:mlmc} in $\varphi$
(same $J_t$, same samples),
\begin{equation*}
  \ghat_t(x_t)=\mlmc{t}[\varphi^{(1)}],
  \qquad
  \widehat\Delta_t=\ghat_t(x_t)-\ghat_t(x_{t-1})=\mlmc{t}[\varphi^{(2)}] .
\end{equation*}
For $\varphi^{(1)}$, $\bar\varphi^{(1)}=\nabla f(x_t)$ and
$\widetilde H\le\Gs$ by \cref{asm:noise}. For $\varphi^{(2)}$,
$\bar\varphi^{(2)}=\nabla f(x_t)-\nabla f(x_{t-1})$, and pointwise in $z$,
\begin{equation*}
\paperfitdisplay{\begin{aligned}
\norm{\varphi^{(2)}(z)-\bar\varphi^{(2)}}
\;\le\;& \norm{\nabla f(x_t;z)-\nabla f(x_{t-1};z)}
        +\norm{\nabla f(x_t)-\nabla f(x_{t-1})}
\;\le\;2L\norm{x_t-x_{t-1}},
\end{aligned}}
\end{equation*}
by individual smoothness (\cref{asm:smooth}, which also implies smoothness
of $f$). Alternatively,
\begin{equation*}
\begin{aligned}
\norm{\varphi^{(2)}(z)-\bar\varphi^{(2)}}
\;\le\;& \norm{\nabla f(x_t;z)-\nabla f(x_t)}
        +\norm{\nabla f(x_{t-1};z)-\nabla f(x_{t-1})}
\;\le\;2\Gs .
\end{aligned}
\end{equation*}
All four bounds of the corollary now follow from
\cref{lem:mlmc}(b)--(c) with the respective $\widetilde H$.
\end{proof}

\subsection{Proof of \texorpdfstring{\cref{thm:reduction}}{the reduction theorem}}

\begin{proof}[Proof of \cref{thm:reduction}]
\emph{Initialization.} At $t=0$, $\alpha_0=1$ and $g_{-1}=0$ give
$g_0^{\mathrm{pre}}=\ghat_0(x_0)$, hence $s_0^{\mathrm{pre}}=n_0$. By
\cref{lem:clipsafe}(c) and \cref{cor:inst}, the following holds almost surely:
\begin{equation*}
\begin{aligned}
\E_{-1}\norm{s_0}^2
\;\le\;&\E_{-1}\norm{n_0}^2
\;\le\;\Lam\Gs^2.
\end{aligned}
\end{equation*}

\emph{Decomposition.} Fix $t\ge1$. From the definition of $\gpre$
(in which $g_{t-1}$ is the \emph{clipped} iterate of the previous step, so
that $s_{t-1}=g_{t-1}-\nabla f(x_{t-1})$ is exactly the clipped error and
is $\F_{t-1}$-measurable),
\begin{equation*}
\begin{aligned}
\spre
\;=\;& (1-\alpha_t)\big(g_{t-1}-\ghat_t(x_{t-1})\big)
      +\ghat_t(x_t)-\nabla f(x_t)\\
\;=\;& (1-\alpha_t)\,s_{t-1}
      +(1-\alpha_t)\,e_t+\alpha_t\,n_t
\;\eqqcolon\;(1-\alpha_t)\,s_{t-1}+\xi_t ,
\end{aligned}
\end{equation*}
by adding and subtracting
$(1-\alpha_t)\nabla f(x_{t-1})$ and $\alpha_t\nabla f(x_t)$.
Write $b_t\coloneqq\E_{t-1}[\xi_t]$.

\emph{Conditional expansion.}
Using \cref{lem:clipsafe}(c) and the $\F_{t-1}$-measurability of
$\alpha_t$ and $s_{t-1}$,
\begin{equation*}
\begin{aligned}
\E_{t-1}\norm{s_t}^2
\;\le\;& \E_{t-1}\norm{\spre}^2\\
\;=\;& (1-\alpha_t)^2\norm{s_{t-1}}^2
      +2(1-\alpha_t)\inner{s_{t-1}}{b_t}
      +\E_{t-1}\norm{\xi_t}^2\\
\cle{1}\;& (1-\alpha_t)^2\norm{s_{t-1}}^2
      +(1-\alpha_t)\Big[\tfrac{\alpha_t}{2}\norm{s_{t-1}}^2
        +\tfrac{2}{\alpha_t}\norm{b_t}^2\Big]
      +\E_{t-1}\norm{\xi_t}^2\\
\cle{2}\;& (1-\alpha_t)\norm{s_{t-1}}^2
      +\tfrac{2}{\alpha_t}\norm{b_t}^2
      +\E_{t-1}\norm{\xi_t}^2 ,
\end{aligned}
\end{equation*}
where
\stepmark{1} is Young's inequality
$2\inner{s_{t-1}}{b_t}\le\tfrac{\alpha_t}{2}\norm{s_{t-1}}^2
+\tfrac{2}{\alpha_t}\norm{b_t}^2$ applied inside the factor
$(1-\alpha_t)$, and
\stepmark{2} uses
$(1-\alpha_t)\big(1-\tfrac{\alpha_t}{2}\big)\le1-\alpha_t$ and
$1-\alpha_t\le1$, valid for all $\alpha_t\in(0,1]$.
No independence between the burst and $s_{t-1}$ is used: the cross term is
controlled by Cauchy--Schwarz--Young pointwise, and the price
$2/\alpha_t$ is paid to the contraction.

\emph{Variance.} By \cref{cor:inst} and $(a+b)^2\le2a^2+2b^2$,
\begin{equation*}
\begin{aligned}
\E_{t-1}\norm{\xi_t}^2
\;\le\;& 2(1-\alpha_t)^2\,\E_{t-1}\norm{e_t}^2
        +2\alpha_t^2\,\E_{t-1}\norm{n_t}^2
\;\le\; 8\Lam L^2\norm{x_t-x_{t-1}}^2+2\alpha_t^2\Lam\Gs^2 .
\end{aligned}
\end{equation*}

\emph{Bias.} Again by \cref{cor:inst} and $(a+b)^2\le2a^2+2b^2$,
\begin{equation*}
\begin{aligned}
\frac{2}{\alpha_t}\norm{b_t}^2
\;\le\;& \frac{2}{\alpha_t}\Big[
   2\Big(\frac{16L\tmix\norm{x_t-x_{t-1}}}{T}\Big)^{2}
  +2\alpha_t^2\Big(\frac{8\Gs\tmix}{T}\Big)^{2}\Big]\\
\;=\;& \frac{4}{\alpha_t}\Big(\frac{16L\tmix\norm{x_t-x_{t-1}}}{T}\Big)^{2}
  +4\alpha_t\Big(\frac{8\Gs\tmix}{T}\Big)^{2}.
\end{aligned}
\end{equation*}
Define $r_t\coloneqq128\tmix^2/(\alpha_t\Lam T^2)$. Each component equals
$r_t$ times its variance sibling, including the zero-noise and
zero-displacement cases:
\begin{equation*}
\begin{aligned}
4\alpha_t^{-1}\Big(\frac{16L\tmix\norm{x_t-x_{t-1}}}{T}\Big)^2
\;=\;&r_t\,8\Lam L^2\norm{x_t-x_{t-1}}^2,\\
4\alpha_t\Big(\frac{8\Gs\tmix}{T}\Big)^2
\;=\;&r_t\,2\alpha_t^2\Lam\Gs^2 .
\end{aligned}
\end{equation*}
The common multiplier is controlled by
\begin{equation*}
\begin{aligned}
\frac{128\,\tmix^2}{\alpha_t\Lam T^2}
\;\cle{1}\;& \frac{128\,\tmix^2(T+1)^{2/3}}{\Lam T^2}
\;\cle{2}\; \frac{128\cdot2^{2/3}\,\tmix^2}{\Lam\,T^{4/3}}\\
\;\cle{3}\;& \frac{128\cdot2^{2/3}}{306}\cdot\frac{\tmix}{T^{4/3}}
\;\le\; 128\,\tmix\,T^{-4/3}
\;\le\;1,
\end{aligned}
\end{equation*}
where
\stepmark{1} is \cref{lem:alphafloor} with $t\le T$.
Step \stepmark{2} uses $(T+1)^{2/3}\le(2T)^{2/3}$.
Step \stepmark{3} uses $\Lam\ge306\tmix$,
and the final step holds precisely when $T^{4/3}\ge128\tmix$, i.e.\
$T\ge(128\tmix)^{3/4}$, which is the burn-in $T\ge T_0$.
(The margin at step \stepmark{3} is in fact a further factor
$128\cdot2^{2/3}/306<0.67$. We keep the conservative threshold.)

\emph{Conclusion.} Each bias component is at most its variance sibling, so
\begin{equation*}
\begin{aligned}
\E_{t-1}\norm{s_t}^2
\;\le\;& (1-\alpha_t)\norm{s_{t-1}}^2
+2\big(8\Lam L^2\norm{x_t-x_{t-1}}^2+2\alpha_t^2\Lam\Gs^2\big)\\
\;=\;& (1-\alpha_t)\norm{s_{t-1}}^2
+2\alpha_t^2\,\sM^2+8\,\LM^2\norm{x_t-x_{t-1}}^2 ,
\end{aligned}
\end{equation*}
with $\sM^2=2\Lam\Gs^2$ and $\LM^2=2\Lam L^2$: the factor $2$ in each
effective constant is exactly the absorption margin for the conditional
bias. Because the recursion is stated conditionally, with all coefficients
$\F_{t-1}$-measurable, and links the clipped errors to themselves (via
\cref{lem:clipsafe}(c) at each step), it can be iterated by the tower
property with no auxiliary unclipped sequence.
\end{proof}

\begin{remark}[Why the momentum recursion, and not an anchored one]
The absorption above leans on the contraction $(1-\alpha_t)$: the bias
enters once per step and is damped geometrically, so paying $2/\alpha_t$
in the Young step costs only the relative factor
$128\tmix T^{-4/3}$. An anchored (SPIDER-type) recursion accumulates the
per-step conditional bias additively over an epoch with no damping. There
is no analogous absorption, and the corresponding composition fails.
This is the structural reason the momentum variant \basevar{} is the right
base for the Markovian extension.
\end{remark}
\section{Proofs for \texorpdfstring{\cref{sec:clipping,sec:res-pathwise}}{the Clipping Sections}: Pathwise Bounds and Transferred Constant Propagation}
\label[appendix]{app:pathwise}

\subsection{Proof of \texorpdfstring{\cref{lem:clipsafe}}{the pathwise-safety lemma}}

\begin{proof}[Proof of \cref{lem:clipsafe}]
\emph{(a).} $\norm{g_t}\le\Ghat$ is the definition of the projection onto
the ball of radius $\Ghat$. Then, by \cref{asm:bounded},
\begin{equation*}
\begin{aligned}
\norm{s_t}
\;\le\;& \norm{g_t}+\norm{\nabla f(x_t)}
\;\le\;\Ghat+G
\;\le\;2\Ghat .
\end{aligned}
\end{equation*}

\emph{(c).} Since $\norm{\nabla f(x_t)}\le G\le\Ghat$, the true gradient is
a fixed point of the projection, and projections onto convex sets are
non-expansive:
\begin{equation*}
\begin{aligned}
\norm{s_t}
\;=\;& \big\lVert\clip(\gpre)-\clip\big(\nabla f(x_t)\big)\big\rVert
\;\le\;\norm{\gpre-\nabla f(x_t)}
\;=\;\norm{\spre}
\qquad\text{pointwise}.
\end{aligned}
\end{equation*}

\emph{(b).} By \cref{asm:problem}, $h$ is $G$-Lipschitz on $\X$. Three
cases.
If $v_t=x_t$, the corner convention gives $\etabar_t=0$, so $\delta_t=0$.
If $\etabar_t<1$, the stepsize attains the unconstrained minimizer, and
\begin{equation*}
\begin{aligned}
\delta_t
\;=\;& L_t\,\etabar_t\,\norm{v_t-x_t}
\;=\;\frac{h(x_t)-h(v_t)-\inner{g_t}{v_t-x_t}}{\norm{v_t-x_t}}\\
\;\le\;& \frac{G\norm{x_t-v_t}+\norm{g_t}\,\norm{v_t-x_t}}{\norm{v_t-x_t}}
\;=\;G+\norm{g_t}.
\end{aligned}
\end{equation*}
If $\etabar_t=1$, the ratio in the stepsize rule is at least $1$, so
\begin{equation*}
\begin{aligned}
L_t\norm{v_t-x_t}^2
\;\le\;& h(x_t)-h(v_t)-\inner{g_t}{v_t-x_t}
\;\le\;\big(G+\norm{g_t}\big)\norm{v_t-x_t}\\
\Longrightarrow\quad
\delta_t\;=\;& L_t\norm{v_t-x_t}\;\le\;G+\norm{g_t}.
\end{aligned}
\end{equation*}
In all cases $\delta_t\le G+\norm{g_t}\le G+\Ghat$ by part (a). We record
the intermediate form for later use:
\begin{equation}\label{eq:lemmaP}
  \delta_t\;\le\;G+\norm{g_t}\;\le\;2G+\norm{s_t}
  \qquad\text{almost surely.}
\end{equation}
\end{proof}

Inequality \cref{eq:lemmaP} explains why a pathwise bound on $\norm{s_t}$
is the essential object: the displacement bound the downstream analysis
invokes is, in substance, an error bound. It also explains why some
algorithmic device is necessary. \Cref{prop:esssup}, proved next, shows
that the raw estimator has no almost-sure bound on $\norm{s_0}$ at any
scale $o(T)\Gs$ exists.

\subsection{Proof of \texorpdfstring{\cref{prop:esssup}}{the heavy-tail proposition}}

\begin{proof}[Proof of \cref{prop:esssup}]
Let $\Z=\{-1,+1\}$ with the symmetric kernel $P(z,-z)=p$ for a fixed
$p\in(0,\tfrac12)$, whose stationary law $\pi$ is uniform. Fix a unit
vector $u\in\R^n$ and a smooth $f_0$ with bounded gradient on $\X$, and set
\begin{equation*}
  f(x;z)\;\coloneqq\;f_0(x)+\Gs\,z\,\inner{u}{x},
\end{equation*}
so that $\nabla f(x;z)=\nabla f_0(x)+\Gs z\,u$, $f=\E_\pi f(\cdot;z)=f_0$
(as $\E_\pi[z]=0$), \cref{asm:smooth} holds with the $L$ of $f_0$, and the
centered noise is exactly $\Gs$. At $t=0$, $\alpha_0=1$ and $g_{-1}=0$ give
$s_0^{\mathrm{pre}}=n_0=\mlmc{0}[\varphi]-\bar\varphi$ with
$\varphi(z)=\nabla f(x_0;z)$.

Let $n\coloneqq2^{\jmax}\ge T/2$ and consider the event
\begin{equation*}
  E\;\coloneqq\;\big\{J_0=\jmax\big\}\;\cap\;
  \big\{z_0^{(1)}=\dots=z_0^{(n/2)}=+1,\;
        z_0^{(n/2+1)}=\dots=z_0^{(n)}=-1\big\}.
\end{equation*}
$E$ has positive probability: $\Prob[J_0=\jmax]=2^{-\jmax}>0$, and every
finite sign pattern has positive probability under $P$ since
$p\in(0,\tfrac12)$. On $E$,
\begin{equation*}
\begin{aligned}
\hat\mu^{\jmax}-\bar\varphi
\;=\;&\Gs u\cdot\frac{1}{n}\Big(\tfrac n2-\tfrac n2\Big)\;=\;0,
\qquad
\hat\mu^{\jmax-1}-\bar\varphi\;=\;\Gs u,
\qquad
\hat\mu^{0}-\bar\varphi\;=\;\Gs u ,
\end{aligned}
\end{equation*}
so that
\begin{equation*}
\begin{aligned}
n_0
\;=\;& (\hat\mu^{0}-\bar\varphi)
   +2^{\jmax}\big(\hat\mu^{\jmax}-\hat\mu^{\jmax-1}\big)
\;=\;\Gs u+n\,(0-\Gs u)
\;=\;-(n-1)\,\Gs\,u ,
\end{aligned}
\end{equation*}
whence $\norm{s_0^{\mathrm{pre}}}=(n-1)\Gs\ge(T/2-1)\Gs$ on $E$.
\end{proof}

The mechanism is the MLMC multiplier $2^{J}$ applied to a level difference
that fails to be small when the burst straddles a long excursion of the
chain. The event is rare, which is why the \emph{second moment} remains
$O(\Lam\Gs^2)$. Essential suprema do not depend on rarity.

\subsection{Transfer of the constant propagation in
\texorpdfstring{\citep{yuan2026alfcg}}{the base analysis}}
\label[appendix]{app:transfer}

We now make precise the claim that the downstream analysis of
\citep[Appendix~H]{yuan2026alfcg} applies to \algname{} under
re-parameterisation. Throughout, ``the base coefficient system'' refers to
(notation of \citep{yuan2026alfcg}, with $C\coloneqq2G$ the Lipschitz
constant of $F$ on $\X$, $A\coloneqq LD+C$, and
$\bar F\coloneqq\sup_{x\in\X}(F(x)-F^\ast)\le2GD$):
\begin{equation}\label{eq:chain}
\begin{gathered}
  \dbar=8\big(A+\sbar+\rho D\big),\quad
  \kappa=\Big(2+\tfrac{2\dbar^2}{\rho}\Big)^{1/2},\quad
  \Bdot=18\big(1+\beta+\dbar^2\big),\quad
  \Bddot=\tfrac{40L^2\kappa^2}{\rho}\ln\big(1+\dbar^2T\big),\\
  Y_1=12D^2\rho^2+144\bar F^2\rho^2,\qquad
  Z_1=27P_0^3,\quad P_0=1+4\bar F\rho+4D^2\rho^2+\tfrac{4L\kappa}{\rho},\\
  Y_2=96\big(\Bdot+\Bddot\big),\qquad
  Z_2=(48\Bddot)^{3/2}+96\Bdot\Big(\sigma^2+\tfrac{\sigma^2}{\beta}\Big),\\
  Y_0=\max(2Y_1,Y_2),\qquad Z_0=\max(2Z_1,Z_2),
\end{gathered}
\end{equation}
with $\vartheta=\beta^{1/6}/(1+\beta^{1/6})$, culminating in
\citep[Lemma~27 and Appendix~H.5]{yuan2026alfcg}:
\begin{equation}\label{eq:master}
  \E\Big[\tfrac{1}{T+1}\textstyle\sum_{t=0}^{T}\gap(x_t)\Big]
  \;\le\;
  \underbrace{\frac{2\,Y_0\max\big(\beta^{1/3},\,\vartheta^2,\,
     \sigma^2\beta^{-2/3}\big)}{\rho\,\vartheta}}_{=:\ \Rcal(\rho,\beta)}
  \,(T+1)^{-1/3}
  \;+\;
  \underbrace{\frac{2\,Z_0}{\rho}}_{=:\ \Scal(\rho,\beta)}\,(T+1)^{-1/2}.
\end{equation}

\begin{proposition}[Transfer]\label{prop:transfer}
Under \cref{asm:problem,asm:smooth,asm:bounded,asm:noise,asm:markov}, for
\algname{} with $\rho>0$, $\beta>0$, and $T\ge T_0$, the bound
\cref{eq:master} holds with the
substitutions
\begin{equation}\label{eq:threesubs}
  \sigma^2\mapsto\sM^2,
  \qquad
  L^2\mapsto\LM^2\ \text{(inside $\Bddot$ only)},
  \qquad
  \sbar\mapsto2\Ghat .
\end{equation}
\end{proposition}

\begin{proof}
The analysis of \citep[Appendix~H]{yuan2026alfcg} consumes the gradient
estimator only through three interfaces, which we verify in turn.

\emph{(i) The descent inequality is pathwise algebra, valid for arbitrary
$g_t$.} \citep[Lemma~6]{yuan2026alfcg} states that for all
$\gamma_t\in(0,1]$,
\begin{equation*}
  \gamma_t\gap(x_t)+F(x_{t+1})-F(x_t)+\frac{\delta_t^2}{4L_t}
  \;\le\;
  \frac{\delta_t^2L}{2L_t^2}+\frac{2\norm{s_t}^2}{L_t}
  +\gamma_t^2L_tD^2
\end{equation*}
(here $\gamma_t$ is a free comparison parameter of the analysis, \emph{not}
the algorithm's step $\etabar_t$). Its proof uses only the $L$-smoothness
of $f$, the convexity of $h$, the optimality of $v_t$ in the composite LMO
\emph{for the same $g_t$ the algorithm used}, and pointwise
Cauchy--Schwarz--Young steps. No unbiasedness, distributional property, or
momentum structure of $g_t$ enters. It therefore holds verbatim for the
clipped $g_t$. (Under the corner convention, if $v_t=x_t$ then
$\delta_t=0$, $x_{t+1}=x_t$, and by LMO optimality at $v=x_t$ the composite
gap evaluated at $g_t$ is zero, so the inequality degenerates to a true
statement.)

\emph{(ii) The momentum identity is consumed at exactly one point.}
In \citep[Appendix~H]{yuan2026alfcg} the update rule for $g_t$ is invoked
once to derive the error recursion \citep[Lemma~26(a)]{yuan2026alfcg}.
Every subsequent step uses the estimator only through $s_t$, $\delta_t$,
and $\alpha_t$. We replace that single derivation by
\cref{thm:reduction}, which is stated conditionally with
$\F_{t-1}$-measurable coefficients and self-links on the clipped sequence
via \cref{lem:clipsafe}(c).  Set $S_t\coloneqq\norm{s_t}^2$,
$S_{-1}\coloneqq0$, and
\begin{equation*}
  \widetilde B_0\coloneqq2\sM^2,
  \qquad
  \widetilde B_t\coloneqq
  2\alpha_t^2\sM^2+8\LM^2\norm{x_t-x_{t-1}}^2
  \quad(1\le t\le T).
\end{equation*}
The initialization in \cref{thm:reduction} gives
$\E_{-1}[S_0]\le\Lam\Gs^2=\sM^2/2\le\widetilde B_0$ almost surely, while
its conditional recursion gives
$\E_{t-1}[S_t]\le(1-\alpha_t)S_{t-1}+\widetilde B_t$ for $t\ge1$.
Therefore \cref{lem:conditionalabel}, with the reciprocal increment from
\cref{lem:alphaincrement}, yields the exact cumulative interface
\begin{equation}\label{eq:conditional-cumulative}
\begin{aligned}
\E\Big[\sum_{t=0}^T\norm{s_t}^2\Big]
\;\le\;&3\E\Big[\sum_{t=0}^T
  \frac{\widetilde B_t}{\alpha_t}\Big]\\
\;=\;&3\E\Big[2\sM^2+
  \sum_{t=1}^T\Big(
    2\alpha_t\sM^2
    +\frac{8\LM^2}{\alpha_t}\norm{x_t-x_{t-1}}^2
  \Big)\Big].
\end{aligned}
\end{equation}
This is precisely the input used after \citep[Eq.~(43)]{yuan2026alfcg}.
Thus, the random predictable weights are transferred by a proved
conditional Abel argument, rather than by treating them as deterministic.

\emph{(iii) The pathwise displacement bound is available by construction.}
The remaining probabilistic ingredient of the analysis is the displacement
bound $\delta_t\le\dbar$, which is invoked along each sample path: in the
stepsize-ratio bound \citep[Lemma~25(c)]{yuan2026alfcg} and in the
logarithmic truncation inside the proof of
\citep[Lemma~26(b)]{yuan2026alfcg}
(where $\ln(1+\sum_{t'\le T-1}\delta_{t'}^2)\le\ln(1+\dbar^2T)$ is applied
pathwise). By \cref{lem:clipsafe}(b),
$\delta_t\le G+\Ghat\le8(A+2\Ghat+\rho D)=\dbar\big|_{\sbar=2\Ghat}$
almost surely, so every such invocation is valid with
$\sbar=2\Ghat$. This bound holds \emph{by construction}, with no
uniform-error hypothesis. (Similarly $\norm{s_t}\le2\Ghat$ a.s.\
strengthens the second-moment role of $\sbar$ wherever it is used.)

Since the chain \cref{eq:chain} is monotone non-decreasing in $\dbar$ and
in $\sigma^2$, and the only occurrences of the recursion constants are the
$\sigma^2$ inherited from \citep[Lemma~26(a)]{yuan2026alfcg} (our $\sM^2$)
and the $L^2$ of $\Bddot$ (traceable to the displacement-variance
coefficient of the recursion, our $\LM^2$), the entire chain and the
master bound \cref{eq:master} hold under \cref{eq:threesubs}. The $L$
appearing in $\dbar$ and in $P_0$ originates in the descent inequality of
step (i), not in the recursion, and is \emph{not} inflated. Tracing this
distinction (\cref{tab:watershed}) determines the resulting $\Lam$-powers.

For completeness, we also make the expectation-level case split in
\citep[Appendix~H.4]{yuan2026alfcg} explicit.  This avoids interpreting a
case distinction made after taking expectations as a pathwise statement.
Set
\begin{equation*}
\begin{gathered}
  X\coloneqq\E[\Gamma],\qquad
  S\coloneqq\E\Big[\sum_{t=0}^{T}\norm{s_t}^2\Big],\qquad
  M\coloneqq\max\big(\beta^{1/3},\vartheta^2,
     \sigma^2\beta^{-2/3}\big),\\
  A_0\coloneqq48\Bdot\sigma^2\Big(1+\frac1\beta\Big)
    +48(\Bdot+\Bddot)M(T+1)^{1/3},
  \qquad b\coloneqq48\Bddot .
\end{gathered}
\end{equation*}
The transferred cumulative-error calculation in
\citep[Appendix~H.3(b)]{yuan2026alfcg}, now justified by
\cref{eq:conditional-cumulative}, gives
\begin{equation*}
\begin{aligned}
  S
  \;\le\;&
  \Bdot\sigma^2\Big(1+\frac1\beta+\beta^{-2/3}T^{1/3}\Big)
  +\Bddot\,\E\big[(\Gamma+T\beta)^{1/3}\big]\\
  \;\le\;&
  \Bdot\sigma^2\Big(1+\frac1\beta\Big)
  +(\Bdot+\Bddot)M(T+1)^{1/3}
  +\Bddot X^{1/3},
\end{aligned}
\end{equation*}
where the second step uses Jensen's inequality and
$(a+b)^{1/3}\le a^{1/3}+b^{1/3}$. Together with the descent calculation
preceding \citep[Eq.~(47)]{yuan2026alfcg}, this yields
\begin{equation}\label{eq:expectation-cases}
\begin{aligned}
  48S &\le A_0+bX^{1/3},\\
  X &\le Z_1+24S+Y_1M(T+1)^{1/3}.
\end{aligned}
\end{equation}
If $24S\le X/2$, then
\begin{equation*}
\begin{aligned}
  \max(X,48S)
  \;\le\;&2Z_1+2Y_1M(T+1)^{1/3}.
\end{aligned}
\end{equation*}
If $24S>X/2$, then $X<48S$ and \cref{eq:expectation-cases} yields
\begin{equation*}
\begin{aligned}
  X
  \;<\;&A_0+bX^{1/3}
  \;\le\;A_0+\frac{X}{3}+\frac23b^{3/2}\\
  \Longrightarrow\quad
  X
  \;\le\;&2A_0+b^{3/2},\\
  48S
  \;\le\;&A_0+bX^{1/3}
  \;\le\;A_0+\frac{X}{3}+\frac23b^{3/2}
  \;\le\;2A_0+b^{3/2}.
\end{aligned}
\end{equation*}
Consequently, in both cases,
\begin{equation*}
\begin{aligned}
  \max(X,48S)
  \;\le\;&
  \max(2Z_1,Z_2)+\max(2Y_1,Y_2)M(T+1)^{1/3}
  \;=\;\Omega,
\end{aligned}
\end{equation*}
where the definitions of $Z_2$ and $Y_2$ in \cref{eq:chain} follow from
$2A_0+b^{3/2}$.  Thus the downstream use of both $\E[\Gamma]$ and the
cumulative estimator error is justified without a random-event case split.

Two elementary proof details, recorded for completeness. First, the
summation inside the logarithmic truncation of
\citep[Lemma~26(b)]{yuan2026alfcg} runs over $t'\le T-1$ (with
$\delta_{-1}=0$), so the truncation constant is $\ln(1+\dbar^2T)$ exactly
as stated. Second, in the case analysis of \citep[Appendix~H.4]{yuan2026alfcg}
we resolve the inequality $x\le A_0+bx^{1/3}$ by Young's inequality,
\begin{equation*}
\begin{aligned}
bx^{1/3}
\;\le\;& \tfrac{x}{3}+\tfrac{2}{3}b^{3/2}
\qquad\Longrightarrow\qquad
x\;\le\;\tfrac{3}{2}A_0+b^{3/2}\;\le\;2A_0+b^{3/2},
\end{aligned}
\end{equation*}
which reproduces the constants $Y_2$, $Z_2$ of \cref{eq:chain}.
\end{proof}

\begin{remark}[On the necessity of the pathwise interface]
The need for an almost-sure interface comes from the adaptive constant
propagation, not from Markovian sampling. The same interface is used in the
i.i.d.\ analysis. For the original estimator of
\citep{yuan2026alfcg}, establishing the bound remains the open challenge
stated in Remark~23. For our raw capped multilevel estimator,
\cref{prop:esssup} rules out a horizon-independent bound at the required
scale. Clipping instead enforces the interface for our algorithm by
construction. It does not settle the original unclipped question.
\end{remark}
\section{Proof for \texorpdfstring{\cref{sec:res-tuned}}{the Tuned-Complexity Section}}\label[appendix]{app:rates}

Throughout this appendix the substitutions \cref{eq:threesubs} are in
force, all constants refer to \cref{eq:chain} so instantiated, and
\begin{equation*}
  \rho^\star=\sqrt{\Lam},\qquad
  \beta^\star=\sM^2=2\Lam\Gs^2,\qquad
  \sbar_1\coloneqq2\Ghat,\qquad
  \Lcal=\max\{1,\ln(1+\dbarM^2T)\},
\end{equation*}
with the standing hypotheses of \cref{thm:tuned}:
$T\ge T_0$ and $0<\Gs\le\overline G_\sigma$, where
$\overline G_\sigma>0$ is the fixed problem-class envelope. Since
$\Lam\ge1$ and the chain
\cref{eq:chain} is monotone in
$\dbar$, and $2\Ghat\le\sqrt{\Lam}\,\sbar_1$ for $\Lam\ge1$, we may
conservatively carry out the constant propagation with
\begin{equation}\label{eq:dbarM}
  \dbarM\;=\;8\big(A+\sqrt{\Lam}\,\sbar_1+\rho^\star D\big),
\end{equation}
which also covers any intermediate constant choice between the
clipped scale $2\Ghat$ and the natural second-moment scale
$\Theta(\sqrt\Lam\,\Gs)$ of the raw error (\cref{prop:inflation}).
The definition of $\Lcal$ takes a maximum with $1$ so that multiplying
log-free terms by $\Lcal$ in Step~6 below is a valid relaxation.

\begin{lemma}[Rate coefficient at the tuned parameters]\label{lem:D}
Under the above,
\begin{equation*}
\begin{aligned}
\Rcal(\rho^\star,\beta^\star)
\;\le\;&
  C_{\Rcal}\,\Lcal\Big(
    \Gs^{1/3}\Lam^{2/3}+\Gs^{2/3}\Lam^{5/6}\Big),\\
C_{\Rcal}
\;\coloneqq\;&2\cdot2^{1/3}
  \Big[24D^2+288\bar F^2+1728c_1+7680c_2L^2\Big].
\end{aligned}
\end{equation*}
where
$c_1\coloneqq1+2\overline G_\sigma^2
  +192\big(A^2+\sbar_1^2+D^2\big)$ and
$c_2\coloneqq2+384\big(A^2+\sbar_1^2+D^2\big)$.
\end{lemma}

\begin{proof}
All steps use $\Lam\ge1$ (hence $\sqrt\Lam\ge1$ and
$\Lam/\sqrt\Lam=\sqrt\Lam$).

\emph{Step 1 ($\dbarM^2$).} By \cref{eq:dbarM} and
$(a+b+c)^2\le3(a^2+b^2+c^2)$,
\begin{equation*}
\begin{aligned}
\dbarM^2
\;\le\;& 192\big(A^2+\Lam\sbar_1^2+\Lam D^2\big)
\;\le\;192\big(A^2+\sbar_1^2+D^2\big)\,\Lam .
\end{aligned}
\end{equation*}

\emph{Step 2 ($\kapM^2$).}
\begin{equation*}
\begin{aligned}
\kapM^2
\;=\;& 2+\frac{2\dbarM^2}{\rho^\star}
\;\le\;2+\frac{384\big(A^2+\sbar_1^2+D^2\big)\Lam}{\sqrt\Lam}
\;\le\;c_2\,\sqrt\Lam .
\end{aligned}
\end{equation*}

\emph{Step 3 ($\BdotM$).}
\begin{equation*}
\begin{aligned}
\BdotM
\;=\;& 18\big(1+\beta^\star+\dbarM^2\big)
\;\le\;18\big(1+2\Lam\overline G_\sigma^2
  +192(A^2+\sbar_1^2+D^2)\Lam\big)
\;\le\;18\,c_1\,\Lam\\
\Longrightarrow\quad
\frac{96\BdotM}{\rho^\star}
\;\le\;& 1728\,c_1\,\sqrt\Lam .
\end{aligned}
\end{equation*}

\emph{Step 4 (absorption of the double inflation in $\BddotM$).}
With $\LM^2=2\Lam L^2$ and $\ln(1+\dbarM^2T)\le\Lcal$,
\begin{equation*}
\begin{aligned}
\BddotM
\;=\;& \frac{40\,\LM^2\,\kapM^2}{\rho^\star}\,\ln\big(1+\dbarM^2T\big)
\;\le\;\frac{80\Lam L^2\cdot c_2\sqrt\Lam\cdot\Lcal}{\sqrt\Lam}
\;=\;80\,c_2L^2\,\Lcal\,\Lam\\
\Longrightarrow\quad
\frac{96\BddotM}{\rho^\star}
\;\le\;& 7680\,c_2L^2\,\Lcal\,\sqrt\Lam .
\end{aligned}
\end{equation*}
Note that $\BddotM$ is only $\Ot(\Lam)$, not $\Ot(\Lam^2)$: the choice
$\rho^\star=\sqrt\Lam$ exactly cancels the $\sqrt\Lam$ of $\kapM^2$.
This cancellation is the entire function of the tuning.

\emph{Step 5 ($Y_1$).}
$\ \dfrac{2Y_1}{\rho^\star}=(24D^2+288\bar F^2)\,\rho^\star
=(24D^2+288\bar F^2)\sqrt\Lam$.

\emph{Step 6 (merge).} Using $Y_0=\max(2Y_1,Y_2)\le2Y_1+96\BdotM+96\BddotM$
and $\Lcal\ge1$,
\begin{equation*}
\begin{aligned}
\frac{Y_0}{\rho^\star}
\;\le\;& \Big[24D^2+288\bar F^2+1728c_1+7680c_2L^2\Big]\,
\Lcal\,\sqrt\Lam
\;\eqqcolon\;C_Y\,\Lcal\,\sqrt\Lam .
\end{aligned}
\end{equation*}

\emph{Step 7 (the max and $\vartheta$).} For every $\beta^\star=\sM^2>0$,
$\vartheta=\beta^{\star1/6}/(1+\beta^{\star1/6})$ and
$\vartheta^2\le\beta^{\star1/3}$, hence
\begin{equation*}
\begin{aligned}
\max\big(\beta^{\star1/3},\,\vartheta^2,\,\sM^2\beta^{\star-2/3}\big)
\;=\;& \max\big(\sM^{2/3},\,\vartheta^2,\,\sM^{2/3}\big)
\;=\;\sM^{2/3}
\;=\;\big(2\Lam\Gs^2\big)^{1/3}.
\end{aligned}
\end{equation*}

\emph{Step 8 (conclusion).}
\begin{equation*}
\begin{aligned}
\Rcal(\rho^\star,\beta^\star)
\;=\;& \frac{2Y_0}{\rho^\star}\,
  \frac{\beta^{\star1/3}}{\vartheta}\\
\;=\;& \frac{2Y_0}{\rho^\star}
  \big(\beta^{\star1/6}+\beta^{\star1/3}\big)\\
\;\le\;&2C_Y\Lcal\Big(
  2^{1/6}\Gs^{1/3}\Lam^{2/3}
  +2^{1/3}\Gs^{2/3}\Lam^{5/6}\Big),
\end{aligned}
\end{equation*}
which is bounded by the claimed expression because $2^{1/6}\le2^{1/3}$.
\end{proof}

\begin{lemma}[Burn-in coefficient at the tuned parameters]\label{lem:S}
Under the same hypotheses, with
$C_P\coloneqq1+4\bar F+4D^2+4L\sqrt{c_2}$,
\begin{equation*}
\begin{aligned}
  6912\,D^6\,\Lam^{5/2}
  \;\le{}&\Scal(\rho^\star,\beta^\star)
  \;\le\;C_{\Scal}\,\Lcal^{3/2}\,\Lam^{5/2},\\
C_{\Scal}
  \;\coloneqq\;&2\Big[54\,C_P^3+(3840\,c_2L^2)^{3/2}
  +1728\,c_1(1+2\overline G_\sigma^2)\Big].
\end{aligned}
\end{equation*}
\end{lemma}

\begin{proof}
\emph{Upper bound.} At $\rho^\star=\sqrt\Lam$, using Step 2 of
\cref{lem:D} ($\kapM\le\sqrt{c_2}\,\Lam^{1/4}$) and $\Lam\ge1$,
\begin{equation*}
\begin{aligned}
P_0
\;=\;& 1+4\bar F\rho^\star+4D^2\rho^{\star2}
   +\frac{4L\kapM}{\rho^\star}
\;\le\;1+4\bar F\Lam+4D^2\Lam+4L\sqrt{c_2}\,\Lam^{-1/4}
\;\le\;C_P\,\Lam ,
\end{aligned}
\end{equation*}
so $2Z_1=54P_0^3\le54\,C_P^3\,\Lam^3$. For $Z_2$, by Steps 3--4 of
\cref{lem:D} and $\beta^\star=\sM^2$,
\begin{equation*}
\begin{aligned}
Z_2
\;=\;& (48\BddotM)^{3/2}
   +96\,\BdotM\big(\beta^\star+1\big)
\;\le\;(3840\,c_2L^2)^{3/2}\,\Lcal^{3/2}\,\Lam^{3/2}
   +1728\,c_1(1+2\overline G_\sigma^2)\,\Lam^{2}.
\end{aligned}
\end{equation*}
Hence, padding every term with $\Lcal^{3/2}\ge1$ and
$\Lam^{3}\ge\Lam^{2}\ge\Lam^{3/2}$,
\begin{equation*}
\begin{aligned}
\Scal
\;=\;& \frac{2Z_0}{\rho^\star}
\;\le\;\frac{2\big(2Z_1+Z_2\big)}{\sqrt\Lam}
\;\le\;C_{\Scal}\,\Lcal^{3/2}\,\Lam^{5/2}.
\end{aligned}
\end{equation*}

\emph{Lower bound.} $P_0\ge4D^2\rho^{\star2}=4D^2\Lam$, so
$Z_1=27P_0^3\ge1728\,D^6\Lam^3$ and
\begin{equation*}
\begin{aligned}
\Scal
\;=\;& \frac{2Z_0}{\rho^\star}
\;\ge\;\frac{4Z_1}{\sqrt\Lam}
\;\ge\;6912\,D^6\,\Lam^{5/2}.
\end{aligned}
\end{equation*}
\end{proof}

When $D>0$, the two bounds in \cref{lem:S} yield
$\Scal=\Tht(\Lam^{5/2})$ up to logarithmic factors. When $D=0$, the upper
bound remains valid, but the displayed lower coefficient vanishes and no
matched-order claim is made. This lower bound is why
\cref{rem:loworder} states the burn-in term
explicitly: at the tuned $\rho^\star$ the constant $Z_1=27P_0^3$ scales as
$\rho^{\star6}=\Lam^3$. This $\rho^6$-dependence drives
$\Scal=\Tht(\Lam^{5/2})$, not the $\Bddot^{3/2}$ term, which is only
$\Tht(\Lam^{3/2})$ here. (The roles reverse in the oblivious regime. See
\cref{app:oblivious}.) Tuning $\rho$ jointly with the target
accuracy can trade the two coefficients against each other, but
\cref{thm:lower} shows the $T^{-1/3}$ coefficient cannot be improved.

\begin{proof}[Proof of \cref{thm:tuned}]
By \cref{prop:transfer} the master bound \cref{eq:master} holds under the
substitutions \cref{eq:threesubs}. The bounds for its two
coefficients follow from \cref{lem:D,lem:S}, giving
\cref{eq:tunedbound}. For the complexity claim, let
$0<\varepsilon\le1$ and choose
\begin{equation*}
  T+1\;\ge\;\max\Big\{\big(2\Rcal/\varepsilon\big)^{3},\;
  \big(2\Scal/\varepsilon\big)^{2},\;T_0\Big\},
\end{equation*}
so that each term of \cref{eq:tunedbound} is at most $\varepsilon/2$, and
therefore
$\E[\gap(x_{\hat t})]=\E\big[\tfrac1{T+1}\sum_t\gap(x_t)\big]\le\varepsilon$
for the uniformly drawn output (and a fortiori
$\E[\min_t\gap(x_t)]\le\varepsilon$). The burn-in term obeys
\begin{equation*}
\begin{aligned}
T_0
\;\le\;&1+(128\tmix)^{3/4}
\;\le\;C_0\Lam^5\varepsilon^{-2},
\end{aligned}
\end{equation*}
for an absolute constant $C_0$, because $\tmix\ge1$,
$\Lam\ge306\tmix$, and $\varepsilon\le1$. Therefore the burn-in horizon
is absorbed. By \cref{lem:D,lem:S} and
$(a+b)^3\le4(a^3+b^3)$ for $a,b\ge0$,
\begin{equation*}
  T\;=\;\Ot\big(\Lam^{2}\,\Gs\,\varepsilon^{-3}
      +\Lam^{5/2}\,\Gs^{2}\,\varepsilon^{-3}
      +\Lam^{5}\,\varepsilon^{-2}\big)
  \;=\;\Ot\big(\tmix^{2}\,\Gs\,\varepsilon^{-3}
      +\tmix^{5/2}\,\Gs^{2}\,\varepsilon^{-3}
      +\tmix^{5}\,\varepsilon^{-2}\big)
\end{equation*}
iterations, and by \cref{lem:mlmc}(d) the expected number of samples is at
most $(1+\log_2 T)$ times the number of iterations, absorbed in
$\Ot(\cdot)$.

Two details of the constant propagation remain. First,
$\Lam=306\tmix(1+\log_2 T)$ depends on $T$
through a logarithm only. Solving the displayed condition for $T$ with
$\Lam=\Lam(T)$ changes the result by polylogarithmic factors, absorbed in
$\Ot(\cdot)$. Second, when $\beta^\star=2\Lam\Gs^2\ge1$, the
$\Lam^{5/2}\Gs^2\varepsilon^{-3}$ term dominates its low-noise companion,
and the original high-effective-noise expression is recovered.
\end{proof}

\begin{lemma}[Exact-oracle transfer]\label{lem:exacttransfer}
If $\Gs=0$, then for every $T\ge1$ the master bound
\cref{eq:master} holds without the burn-in condition $T\ge T_0$, with
$\sigma^2=0$, $\sbar=0$, and with the occurrence of $L^2$ inside
$\Bddot$ replaced by $0$.  The smoothness constant $L$ in $P_0$ remains
unchanged.
\end{lemma}

\begin{proof}
$\Gs=0$ implies $\nabla f(x;z)=\nabla f(x)$ for every $(x,z)$. Hence all
levels of the MLMC estimator at a fixed point coincide, and
$\ghat_t(x)=\nabla f(x)$ pointwise. Inductively, $g_{-1}=0$ and
$\alpha_0=1$ give $g_0^{\mathrm{pre}}=\nabla f(x_0)$. If
$g_{t-1}=\nabla f(x_{t-1})$, then
\begin{equation*}
\begin{aligned}
g_t^{\mathrm{pre}}
\;=\;&(1-\alpha_t)
  \big(\nabla f(x_{t-1})-\nabla f(x_{t-1})\big)+\nabla f(x_t)
\;=\;\nabla f(x_t).
\end{aligned}
\end{equation*}
Since $\norm{\nabla f(x_t)}\le G\le\Ghat$, clipping is inactive and
$s_t=0$ for all $t$. The three transfer interfaces in the proof of
\cref{prop:transfer} therefore hold with the exact recursion
$\E_{t-1}\norm{s_t}^2=0$, whose noise and displacement coefficients are
both zero, and with the pathwise scale $\sbar=0$. The condition
$T\ge T_0$ enters only when the Markovian conditional bias is absorbed in
\cref{thm:reduction}. Here that bias is identically zero. The downstream
deterministic constant propagation thus gives \cref{eq:master} for
every $T\ge1$ with the stated specialization.
\end{proof}

\begin{proof}[Proof of \cref{cor:zeronoise}]
Run \algname{} with fixed $\rho=\rho_0=\Theta(1)$ and
$\beta=\beta_T\coloneqq(T+1)^{-1}>0$. In the transferred constant
propagation, the
exact-oracle transfer of \cref{lem:exacttransfer} sets both
recursion-only coefficients to zero, so $\Bddot=0$ and
\begin{equation*}
  Y_0=O(1),\qquad Z_0=O(1),\qquad
  \vartheta=\frac{\beta_T^{1/6}}{1+\beta_T^{1/6}}.
\end{equation*}
Moreover,
\begin{equation*}
\begin{aligned}
\frac{\max(\beta_T^{1/3},\vartheta^2,0)}{\vartheta}
\;=\;&\frac{\beta_T^{1/3}}{\vartheta}
\;=\;\beta_T^{1/6}+\beta_T^{1/3}
\;=\;O\big((T+1)^{-1/6}\big).
\end{aligned}
\end{equation*}
Substitution into \cref{eq:master} gives
\begin{equation*}
\E\Big[\frac1{T+1}\sum_{t=0}^T\gap(x_t)\Big]
\;=\;\Ot\big((T+1)^{-1/2}\big),
\end{equation*}
and therefore $\Ot(\varepsilon^{-2})$ expected samples after including
the logarithmic expected burst cost. No constant in this specialization
depends on $\tmix$.
\end{proof}
\section{Proofs for
\texorpdfstring{\cref{sec:res-lower}}{the Lower-Bound Section}:
An In-Analysis Exponent Bound and Pre-Clipping Inflation}
\label[appendix]{app:lower}

\subsection{Proof of \texorpdfstring{\cref{thm:lower}}{the lower bound}}

\begin{proof}[Proof of \cref{thm:lower}]
Fix any $\rho>0$, $\beta>0$, and any $\dbar\ge0$ substituted into
\cref{eq:chain}. The proof uses only two inequalities in the transferred
constant propagation. Both are independent of $\dbar$, which makes the
bound uniform over every admissible pathwise bound:
\begin{equation*}
\begin{aligned}
Y_0
\;\ge\;& 2Y_1\;=\;24D^2\rho^2+288\bar F^2\rho^2\;\ge\;24D^2\rho^2
&&\Longrightarrow\quad \frac{Y_0}{\rho}\;\ge\;24D^2\rho\,,\\
Y_0
\;\ge\;& Y_2\;\ge\;96\Bdot\;=\;96\cdot18\big(1+\beta+\dbar^2\big)
\;\ge\;1728\,\beta
&&\Longrightarrow\quad \frac{Y_0}{\rho}\;\ge\;\frac{1728\,\beta}{\rho}\,.
\end{aligned}
\end{equation*}
Hence
\begin{equation*}
\begin{aligned}
\frac{Y_0}{\rho}
\;\ge\;& \max\Big(24D^2\rho,\;\frac{1728\,\beta}{\rho}\Big)
\;\cge{1}\;\sqrt{24\cdot1728}\;D\,\sqrt{\beta}
\;=\;144\sqrt2\;D\,\sqrt{\beta}\,,
\end{aligned}
\end{equation*}
where \stepmark{1} is $\max(a,b)\ge\sqrt{ab}$
($24\cdot1728=41472=2\cdot144^2$). Using $\vartheta\le1$ in
\cref{eq:master},
\begin{equation*}
\paperfitdisplay{\begin{aligned}
\Rcal(\rho,\beta)
\;\ge\;& 2\,\frac{Y_0}{\rho}\,
   \max\big(\beta^{1/3},\,\sM^2\beta^{-2/3}\big)
\;\ge\;288\sqrt2\;D\;g(\beta),
\qquad
g(b)\coloneqq\sqrt b\,\max\big(b^{1/3},\,\sM^2 b^{-2/3}\big).
\end{aligned}}
\end{equation*}
On $b\le\sM^2$ we have $b^{1/3}\le\sM^2b^{-2/3}$, so
$g(b)=\sM^2\,b^{-1/6}$, strictly decreasing. On $b\ge\sM^2$,
$g(b)=b^{5/6}$, strictly increasing. Hence
\begin{equation*}
\begin{aligned}
\min_{b>0}g(b)
\;=\;& g\big(\sM^2\big)
\;=\;\sM\cdot\big(\sM^2\big)^{1/3}
\;=\;\sM^{5/3},
\end{aligned}
\end{equation*}
and $\Rcal(\rho,\beta)\ge288\sqrt2\,D\,\sM^{5/3}
=288\sqrt2\cdot2^{5/6}D\,\Gs^{5/3}\Lam^{5/6}$ for all $(\rho,\beta)$.
Combining with \cref{lem:D} gives the squeeze
\begin{equation*}
\begin{aligned}
  288\sqrt2\cdot2^{5/6}\,D\,\Gs^{5/3}\,\Lam^{5/6}
  \;\le\;
  \min_{\rho,\beta}\Rcal(\rho,\beta)
  \;\le\;
  \Rcal(\rho^\star,\beta^\star)
  \;\le\;
  C_{\Rcal}\,\Lcal
  \big(\Gs^{1/3}\Lam^{2/3}+\Gs^{2/3}\Lam^{5/6}\big).
\end{aligned}
\end{equation*}
For fixed $D>0$ and $\Gs>0$, the $\Lam^{5/6}$ term dominates the tuned
upper bound as $\Lam\to\infty$.  Thus the upper and lower bounds have the
same asymptotic $\Lam$-exponent, up to
$\Lcal=O(\log(\Lam T))$.  The comparison is not uniform in $\Gs$, and
does not identify the exact minimizer of the full coefficient.
\end{proof}

\begin{remark}[What the surrogate balance certifies]
The minimizer of $g$ is uniquely $b=\sM^2$: the parameter $\beta$ appears
in the numerator through $\Bdot$ and in the denominator through the
stepsize schedule, and the two branches cross exactly at $\beta=\sM^2$.
This lower-bound surrogate motivates the tuning
$\beta^\star=\sM^2$, consistent with \citep[Remark~29]{yuan2026alfcg}.
It does not prove that $\beta^\star$ exactly minimizes the full
$\Rcal(\rho,\beta)$, whose remaining terms also depend on the parameters.
\end{remark}

\subsection{Proof of \texorpdfstring{\cref{prop:inflation}}{the inflation
proposition}}

\begin{proof}[Proof of \cref{prop:inflation}]
\emph{Instance.} As in \cref{prop:esssup}, take $\Z=\{-1,+1\}$,
$P(z,-z)=p$ with $p\in(0,\tfrac14]$, stationary (uniform) initialization,
and $\nabla f(x;z)=\nabla f_0(x)+\Gs z\,u$ for a unit vector $u$, so that
at $t=0$ (where $\alpha_0=1$, $g_{-1}=0$) the pre-clip error is the scalar
MLMC error: $s_0^{\mathrm{pre}}=n_0=\big(\mlmc{0}[\varphi]-\bar\varphi\big)$
with $\varphi(z)=\Gs z\,u+\nabla f_0(x_0)$. Write $r\coloneqq1-2p$ and
$q\coloneqq2p$. Under stationarity
$\Cov\big(z^{(i)},z^{(k)}\big)=r^{\abs{i-k}}$.

\emph{Exact segment moments.} For the averages
$\hat A\coloneqq\frac1n\sum_{i=1}^{n}z^{(i)}$ and
$\hat B\coloneqq\frac1n\sum_{i=n+1}^{2n}z^{(i)}$ over two adjacent
segments of length $n$, summing the geometric covariances gives
\begin{equation*}
  n^2\,\Var(\hat A)
  \;=\;\frac{n(1+r)}{q}-\frac{2r(1-r^{n})}{q^{2}},
  \qquad
  n^2\,\Cov(\hat A,\hat B)
  \;=\;\frac{r\,(1-r^{n})^{2}}{q^{2}}\;>\;0 .
\end{equation*}
With $n=2^{j-1}$, $\hat\mu^{j}-\hat\mu^{j-1}=\tfrac12(\hat B-\hat A)\Gs u$
and $\Var(\hat A)=\Var(\hat B)$ by stationarity, so
\begin{equation*}
  V_j
  \;=\;\frac{\Gs^2}{2}\Big(\Var(\hat A)-\Cov(\hat A,\hat B)\Big).
\end{equation*}
The two halves are \emph{positively} correlated, and the covariance
\emph{reduces} $V_j$. It must be controlled explicitly, not ignored.

\emph{Controlling the covariance.} From $(1-r^{n})^{2}\le1$,
$n^2\Cov(\hat A,\hat B)\le r/q^{2}$, and the same relaxation bounds the
second term of $\Var(\hat A)$, so
\begin{equation*}
\begin{aligned}
n^2\big(\Var(\hat A)-\Cov(\hat A,\hat B)\big)
\;\ge\;& \frac{n(1+r)}{q}-\frac{3r}{q^{2}}
\;\cge{1}\;\frac{n(1+r)}{4q}
\;\ge\;\frac{n}{4q},
\end{aligned}
\end{equation*}
where \stepmark{1} holds whenever $n\ge2/q=1/p$: then
$n(1+r)\ge\tfrac{2}{q}(1+r)\ge\tfrac{4r}{q}$ (using $r\le1$), so
$\tfrac{3r}{q^2}\le\tfrac34\cdot\tfrac{n(1+r)}{q}$. Consequently, for
every level with $2^{j-1}\ge1/p$,
\begin{equation*}
\begin{aligned}
2^{j}V_j
\;=\;& 2n\,V_j
\;\ge\;2n\cdot\frac{\Gs^2}{2}\cdot\frac{1}{4qn^{2}}\cdot n
\;=\;\frac{\Gs^2}{4q}
\;=\;\frac{\Gs^2}{8p}.
\end{aligned}
\end{equation*}

\emph{Conversion to the mixing time.} For this chain
$\dmix(k)=\tfrac12r^{k}$, so
$\tmix=\lceil\ln2/\ln(1/r)\rceil$. For $p\in(0,\tfrac14]$ one has
$2p\le\ln(1/r)\le4p$, whence
\begin{equation*}
  \tmix\;\le\;\frac1p\;\le\;6\,\tmix .
\end{equation*}
Thus every valid level contributes $2^{j}V_j\ge\Gs^2/(8p)\ge\Gs^2\tmix/8$,
and validity $2^{j-1}\ge1/p$ is implied by $2^{j}\ge16\tmix$
(since $2/p\le12\tmix$). The number of levels $j$ with
$16\tmix\le2^{j}\le2^{\jmax}$ is at least $\log\big(T/(64\tmix)\big)$
(using $2^{\jmax}\ge T/2$ and generous rounding).

\emph{Assembling.} Write $X\coloneqq\hat\mu^{0}-\bar\varphi$
(so $\norm{X}\le\Gs$ pointwise) and
$C_J\coloneqq2^{J_0}(\hat\mu^{J_0}-\hat\mu^{J_0-1})\ind{J_0\le\jmax}$.
From $\norm{C_J}^2\le2\norm{X+C_J}^2+2\norm{X}^2$ and
$J_0\perp$ chain,
\begin{equation*}
\paperfitdisplay{\begin{aligned}
\E\norm{s_0^{\mathrm{pre}}}^2
\;=\;\E\norm{X+C_J}^2
\;\ge\;& \tfrac12\,\E\norm{C_J}^2-\E\norm{X}^2
\;=\;\tfrac12\sum_{j\le\jmax}2^{-j}\,4^{j}\,V_j-\E\norm{X}^2\\
\;\ge\;& \tfrac12\sum_{j:\,16\tmix\le2^{j}\le2^{\jmax}}2^{j}V_j-\Gs^2
\;\ge\;\frac{\Gs^2\tmix}{16}\,\log\!\Big(\frac{T}{64\tmix}\Big)-\Gs^2 .
\end{aligned}}
\end{equation*}
Finally, if $T\ge(64\tmix)^2$ then $T/(64\tmix)\ge\sqrt T$, so
$\log(T/(64\tmix))\ge\tfrac12\log T$ and the bound is
$\ge\tfrac{\Gs^2\tmix\log T}{32}-\Gs^2=\Omega(\Lam\Gs^2)$ up to absolute
constants. (For the displayed constant to be positive at finite $T$ one
additionally needs $\tmix\log(T/(64\tmix))\ge32$. We use the bound in its
asymptotic reading.)
\end{proof}

\begin{remark}[What \cref{prop:inflation} certifies]
The proposition shows that the \emph{pre-clipping} MLMC error can attain
the scale $\Theta(\Lam\Gs^2)$, up to absolute constants.  It therefore
validates the raw-estimator moment bound used by the black-box reduction.
It does not imply that the clipped error has the same lower bound:
$\norm{s_t}\le\norm{s_t^{\mathrm{pre}}}$ and
$\norm{s_t}\le2\Ghat$ pathwise.  A clipping-aware recursion could thus
improve the effective-noise coefficient. \Cref{thm:lower} applies only to
the transferred constant propagation with the substitution
\cref{eq:substitution}.
\end{remark}
\section{Proof for \texorpdfstring{\cref{sec:res-oblivious}}{the Oblivious-Complexity Section}}\label[appendix]{app:oblivious}

\begin{proof}[Proof of \cref{thm:oblivious}]
Fix $\rho=\rho_0$ and $\beta=\beta_0$, constants independent of $\tmix$
(for concreteness $\rho_0=\beta_0=1$, while general $\Theta(1)$ values only
change the constants below). The clipping radius $\Ghat$ is likewise
$\tmix$-free, so the algorithm as a whole uses no mixing-time knowledge.
$\Lam$ appears in the \emph{analysis} only.
Since $\beta_0$ is not tied to $\sM^2$, we do not pass through the tuned
specialization of \citep[Theorem~28]{yuan2026alfcg}, which pre-substitutes
$\beta\propto\sigma^2$. We instantiate the general chain
\cref{eq:chain}--\cref{eq:master} of \cref{prop:transfer} directly.

Because clipping gives the $\Lam$-free pathwise scale
$\sbar=2\Ghat$, the front of the chain does not inflate at all:
\begin{equation*}
\paperfitdisplay{\begin{aligned}
\dbar
\;=\;& 8\big(A+2\Ghat+\rho_0D\big)\;=\;\Theta(1),
\qquad
\kappa^2\;=\;2+2\dbar^2\;=\;\Theta(1),
\qquad
\Bdot\;=\;18\big(2+\dbar^2\big)\;=\;\Theta(1),
\end{aligned}}
\end{equation*}
and $\Lam$ enters exactly once, through the recursion coefficient in
$\Bddot$:
\begin{equation*}
\begin{aligned}
\Bddot
\;=\;& 40\,\LM^2\,\kappa^2\,\ln\big(1+\dbar^2T\big)
\;=\;80\Lam L^2\kappa^2\ln\big(1+\dbar^2T\big)
\;=\;\Tht(\Lam),
\end{aligned}
\end{equation*}
whence $Y_1=\Theta(1)$, $Y_2=96(\Bdot+\Bddot)=\Ot(\Lam)$, and
$Y_0=\Ot(\Lam)$. For the stepsize block, $\vartheta=\Theta(1)$ and
\begin{equation*}
\begin{aligned}
\max\big(\beta_0^{1/3},\,\vartheta^2,\,\sM^2\beta_0^{-2/3}\big)
\;\le\;& C_0\big(1+\sM^2\big)
\;=\;O\big(1+\Lam\Gs^2\big),
\end{aligned}
\end{equation*}
where $C_0$ depends only on the fixed $\beta_0$.
Therefore
\begin{equation*}
\begin{aligned}
\Rcal
\;=\;& \frac{2Y_0\max(\cdot)}{\rho_0\vartheta}
\;=\;\Ot\big(\Lam+\Lam^{2}\Gs^2\big),
\end{aligned}
\end{equation*}
and for the burn-in block,
\begin{equation*}
\begin{aligned}
Z_1
\;=\;& 27P_0^3\;=\;\Theta(1),\\
Z_2\;=\;& (48\Bddot)^{3/2}
   +96\Bdot\,\sM^2\Big(1+\tfrac1{\beta_0}\Big)
\;=\;\Ot\big(\Lam^{3/2}+\Lam\Gs^2\big),
\end{aligned}
\end{equation*}
so
$Z_0=\Ot(\Lam^{3/2}+\Lam\Gs^2)$ and
$\Scal=2Z_0/\rho_0=\Ot(\Lam^{3/2}+\Lam\Gs^2)$.
Solving $\Rcal(T+1)^{-1/3}\le\varepsilon/2$ and
$\Scal(T+1)^{-1/2}\le\varepsilon/2$ as in the proof of \cref{thm:tuned}
and multiplying by the $O(\log T)$ expected samples per iteration
(\cref{lem:mlmc}(d)) gives
\begin{equation*}
\begin{aligned}
  T
  \;=\;&\Ot\big(\Rcal^3\varepsilon^{-3}+\Scal^2\varepsilon^{-2}\big)\\
  \;=\;&\Ot\Big(
    (\Lam^3+\Lam^6\Gs^6)\varepsilon^{-3}
    +(\Lam^3+\Lam^2\Gs^4)\varepsilon^{-2}
  \Big)\\
  \;=\;&\Ot\big(\Lam^{6}\,\varepsilon^{-3}+\Lam^{3}\,\varepsilon^{-2}\big)
  \;=\;\Ot\big(\tmix^{6}\,\varepsilon^{-3}
      +\tmix^{3}\,\varepsilon^{-2}\big).
\end{aligned}
\end{equation*}
The penultimate simplification treats the finite $\Gs$ as a fixed problem
constant. The preceding line records the non-uniform dependence explicitly.
\end{proof}

\begin{remark}[Where the regimes differ]
In the oblivious upper bound the potentially largest burn-in contribution is the
$(48\Bddot)^{3/2}=\Ot(\Lam^{3/2})$ term, while $Z_1=\Theta(1)$. In the
tuned regime of \cref{app:rates} the roles are exactly reversed
($Z_1=\Theta(\Lam^3)$ via $\rho^{\star6}$, while
$(48\BddotM)^{3/2}=\Ot(\Lam^{3/2})$). The dominant constant depends on the
parameter regime. Tracking only one contribution would give an incorrect
bound in one of the two regimes.
\end{remark}

\begin{remark}[Effect of parameter independence]
Relative to the tuned theorem, the cost of obliviousness is regime
dependent. The tuned high-effective-noise branch scales as
$\tmix^{5/2}\Gs^2\varepsilon^{-3}$, while its low-effective-noise companion
scales as $\tmix^2\Gs\varepsilon^{-3}$. The oblivious stochastic term scales
as $\tmix^6\varepsilon^{-3}$ after fixed noise constants are suppressed.
The mixing-time-oblivious baseline of
\citep{dorfman2022adapting} is non-composite, not projection-free, and does
not use variance reduction, so its $\Ot(\varepsilon^{-4})$ rate is not a
direct comparator. We do not claim simultaneous obliviousness and the
tuned high-effective-noise exponent.
\end{remark}
\section{Experiments}\label[appendix]{sec:experiments}

The first experiment addresses a focused question: how does each method
deteriorate as temporal dependence increases while the stationary objective
remains fixed? We reproduce the low-rank multiclass-classification setting
used for the base method \citep{yuan2026alfcg} and replace only the
independent data stream by a family of Markov chains with a common
stationary distribution. The study assesses robustness to dependence. It
does not establish large-scale empirical superiority or verify an
asymptotic exponent.

\subsection{Datasets and Tasks}\label{sec:exp-data}

We minimize $f(X)=\E_{z\sim\pi}[f(X;z)]$, the multinomial logistic loss,
over $\X=\{X\in\R^{d\times K}:\norm{X}_{*}\le r\}$ with $r=10$ (the LMO
returns one dominant singular pair and $h=0$). The data are a synthetic
Gaussian mixture in the spirit of the \texttt{randn} instances of
\citep{yuan2026alfcg}, with $n=1000$ unit-normalized points in $\R^{50}$
and $K=10$ classes. To isolate mixing from the objective, every experiment
uses the same uniform stationary law $\pi_i=1/n$. Only the transition
kernel changes:
\begin{equation*}
  P_q=(1-q)I+q\,\bm{1}\bm{1}^{\mathsf T}/n .
\end{equation*}
This lazy-refresh family is symmetric and doubly stochastic, hence
reversible with respect to the common $\pi$. Moreover, its worst-case
total-variation distance is available in closed form,
$d_{\rm mix}(k)=(1-q)^k(1-1/n)$. We choose $q$ so that the mixing time in
\cref{asm:markov} is exactly
$\tmix\in\{1,\,10,\,100,\,334\}$. Thus the four columns optimize the
\emph{same} finite-sum objective and differ only in temporal dependence.

\subsection{Compared Methods}\label{sec:exp-methods}

We compare four methods.
(i) \algname{} \emph{mixing-aware}: $\rho=\rho_0\sqrt{\hat\Lam}$ and
$\beta=2\hat\Lam\bar G_\sigma^2$, where
$\hat\Lam=\tmix(1+\lfloor\log_2 T\rfloor)$ and
$\bar G_\sigma=2\sqrt2\max_i\norm{a_i}$ is a computable upper bound on
the centered gradient noise. This follows only the parameter \emph{scaling}
prescribed by \cref{thm:tuned}, namely
$\rho\propto\sqrt{\tmix\log T}$ and
$\beta\propto\tmix\log T\,\bar G_\sigma^2$. It does not use the theorem's
absolute constants and is therefore not presented as the theorem-level tuned
instantiation.
(ii) \algname{} \emph{oblivious}: $(\rho,\beta)=(\rho_0,2\bar
G_\sigma^2)$, with no mixing input.
(iii) The base \basevar{} is fed the Markovian stream as if it were i.i.d.\
using single-sample batches, no MLMC, and no clipping. This ablation isolates
the estimator. (iv) Projected SGD follows the Markovian method in
\citep{alacaoglu2023convergence}, step $c\,D/(\Ghat\sqrt{t+1})$.
The calibration values $\rho_0$ and $c$ are selected on independent i.i.d.\
paths and then fixed across all four mixing settings. Projected SGD is
included as a statistical baseline rather than as a projection-cost
comparison. The experiment runs the clipped algorithm exactly as stated in
\cref{alg:main}, including the corner convention
$v_t=x_t\Rightarrow\etabar_t=0$.

\subsection{Parameter Settings}\label{sec:exp-params}

\algname{} is run for exactly $T+1=8301$ iterations, as in
\cref{alg:main}. With $j_{\max}=13$, its expected sample use is
$(T+1)(j_{\max}+2^{-j_{\max}})\approx1.08\times10^5$. The realized total
varies across capped-MLMC draws and is reported below. The sample-based
baselines run for $1.08\times10^5$ updates. The burn-in condition
$T\ge T_0$ of \cref{thm:reduction} holds in all four settings. Results
aggregate $10$ seeds, with the same ten random-number seeds used in every
mixing column. On five independent i.i.d.\ calibration paths (seeds
$100$--$104$, disjoint from the reporting seeds $0$--$9$), a grid search
selects $\rho_0=10^{-1}$ from
$\{10^{-3},3\!\times\!10^{-3},10^{-2},3\!\times\!10^{-2},10^{-1}\}$
for the base method and $c=10^{-1}$ from $\{10^{-1},1,10\}$ for projected
SGD\@. Both are then fixed across all reported settings. The naive base
ablation uses the base-paper default $\beta=100$. The clipping
radius is the rigorous a-priori bound
$\Ghat=\sqrt2\max_i\norm{a_i}\ge G$, independent of $\tmix$ and $T$.
Across the ten reported seeds, the capped-MLMC methods use
$(1.129\pm0.099)\times10^5$ samples (range
$9.61\times10^4$--$1.35\times10^5$) over their fixed horizon.

\begin{figure}[t]
\centering
\includegraphics[width=\linewidth]{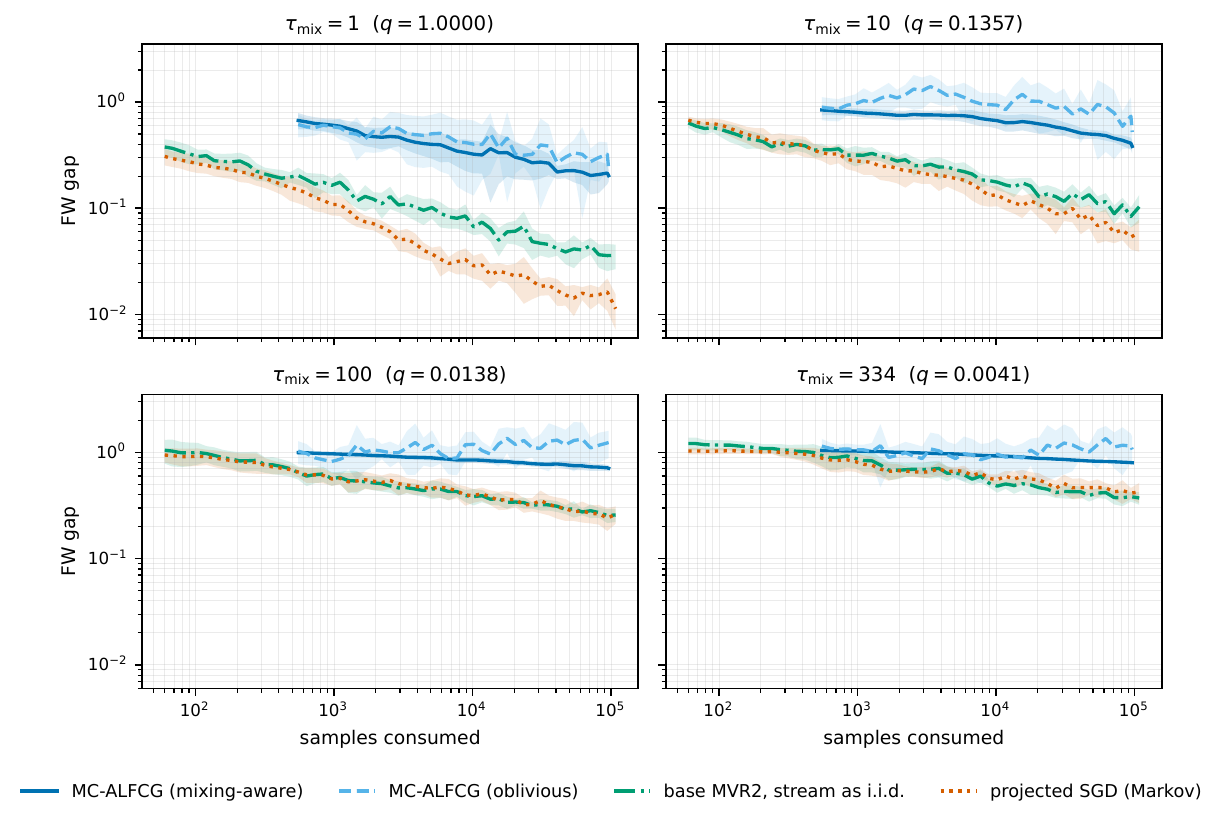}
\caption{FW gap \cref{eq:fwgap} versus samples consumed for four kernels
with a common uniform stationary law and different exact mixing times.
Lines are means over $10$ seeds. Bands are $\pm1$ standard deviation
(clipped below at one quarter of the mean for the logarithmic display).}
\label{fig:fwgap}
\end{figure}

\begin{table}[t]
\centering
\caption{Final-iterate FW gap (mean $\pm$ standard deviation over $10$
seeds). \algname{} uses a fixed horizon with expected
cost $1.08\times10^5$ samples. Baselines use $1.08\times10^5$ updates.
Middle block: mixing-aware variant run with inputs
$\max\{\lfloor\tmix/4\rfloor,1\}$ and $4\tmix$ (at $\tmix=1$ the former
coincides with the true value by construction).
Bottom block: ratio to each method's own i.i.d.\ value.}
\label{tab:final}
\small
\ifdefined\JOTAmode
\resizebox{\linewidth}{!}{%
\fi
\begin{tabular}{lcccc}
\toprule
 & $\tmix=1$ & $\tmix=10$ & $\tmix=100$ & $\tmix=334$ \\
\midrule
\algname{} (mixing-aware) & $0.183\pm0.039$ & $0.360\pm0.045$ & $0.688\pm0.046$ & $0.784\pm0.029$ \\
\algname{} (oblivious) & $0.205\pm0.046$ & $0.617\pm0.269$ & $1.023\pm0.281$ & $1.082\pm0.364$ \\
\basevar{}, stream as i.i.d. & $0.036\pm0.009$ & $0.103\pm0.030$ & $0.258\pm0.037$ & $0.372\pm0.051$ \\
projected SGD & $0.011\pm0.004$ & $0.059\pm0.020$ & $0.261\pm0.049$ & $0.423\pm0.091$ \\
\midrule
mixing-aware, $\tmix/4$ & $0.183\pm0.039$ & $0.339\pm0.061$ & $0.612\pm0.063$ & $0.651\pm0.050$ \\
mixing-aware, $4\tmix$ & $0.181\pm0.025$ & $0.461\pm0.043$ & $0.821\pm0.032$ & $0.917\pm0.016$ \\
\midrule
degradation, mixing-aware & $1.0\times$ & $2.0\times$ & $3.8\times$ & $4.3\times$ \\
degradation, oblivious & $1.0\times$ & $3.0\times$ & $5.0\times$ & $5.3\times$ \\
degradation, base as i.i.d. & $1.0\times$ & $2.9\times$ & $7.2\times$ & $10.3\times$ \\
degradation, projected SGD & $1.0\times$ & $5.2\times$ & $23.1\times$ & $37.4\times$ \\
\bottomrule
\end{tabular}
\ifdefined\JOTAmode
}%
\fi
\end{table}

\subsection{Results and Discussion}\label{sec:exp-results}

\emph{(a) Dependence sensitivity and absolute performance.}
Because all four chains share the same stationary distribution, variation
across columns reflects temporal dependence rather than a change in the
objective. As $\tmix$ increases from $1$ to $334$, the final gaps of the
mixing-aware and oblivious variants increase by factors of $4.3$ and $5.3$,
compared with $10.3$ for the naive base estimator and $37.4$ for projected
SGD (\cref{tab:final,fig:fwgap}). The multilevel variants therefore exhibit
the smallest relative degradation on this testbed. Their absolute final
gaps are nevertheless not the smallest: projected SGD performs best for
$\tmix\in\{1,10\}$, whereas the naive base method performs best for
$\tmix\in\{100,334\}$. These results support a robustness-to-dependence
interpretation, but not absolute empirical superiority.

The table compares the methods at their prescribed terminal iterations.
The multilevel methods use a random realized sample count, whereas the
baselines use the common expected budget $1.08\times10^5$. Consequently,
the table is not an exactly budget-matched comparison. The sample-resolved
curves in \cref{fig:fwgap} provide the appropriate within-budget view.

\emph{(b) Clipping frequency.} Across all
\algname{} runs, the clipping activation frequency
$\#\{t:\norm{\gpre}>\Ghat\}/(T+1)$ ranges from $1.20\times10^{-4}$ to
$3.61\times10^{-3}$ (mean $1.60\times10^{-3}$). Thus clipping is rarely
active on this testbed, while the pathwise bound it enforces remains a
necessary ingredient of the analysis.

\emph{(c) Sensitivity to the mixing input.} The $4\tmix$ input increases
the final gap in the dependent settings, whereas the
$\max\{\lfloor\tmix/4\rfloor,1\}$ input decreases it at this finite horizon
(\cref{tab:final}, middle block). Thus the theorem-level tuning should not
be interpreted as a finite-horizon constant calibration. The oblivious
variant removes the mixing-time input altogether at the exponent cost
quantified in \cref{thm:oblivious}.

\subsection{Nonconvex Composite Regression}\label{sec:exp-composite}

The second experiment exercises the full model rather than the smooth convex
special case of the classification study. For unit-norm
$a_i\in\R^{30}$, we consider
\begin{equation}\label{eq:composite-instance}
\min_{x\in[-3,3]^{30}}
\left\{
\Phi(x)
\coloneqq
\frac{1}{2n}\sum_{i=1}^{n}
  \big(\sin(a_i^\top x)-b_i\big)^2
+0.02\norm{x}_1
\right\},
\qquad n=800.
\end{equation}
The features are normalized Gaussian vectors. A five-sparse vector with
nonzero entries in $\{-2,2\}$ generates
$b_i=\operatorname{clip}(\sin(a_i^\top x^\star)+\xi_i,-1,1)$, where
$\xi_i\sim\mathcal N(0,0.1^2)$. The data seed is $19$. A deterministic
verification found a Hessian eigenvalue $-0.0200$ for the fixed finite-sum
smooth objective, while its directional-gradient finite-difference error was
$1.43\times10^{-11}$. Thus the instance numerically exercises smooth
nonconvexity, and $h(x)=0.02\norm{x}_1$ is nonzero, convex, and nonsmooth.

For a gradient vector $g$, the generalized LMO is separable:
\begin{equation*}
v_j(g)
=
\begin{cases}
-3, & g_j>0.02,\\
0, & |g_j|\le 0.02,\\
3, & g_j<-0.02.
\end{cases}
\end{equation*}
Moreover,
$|\sin(a_i^\top x)-b_i|\,|\cos(a_i^\top x)|\le2$, so
$\Ghat=2$ and $\bar G_\sigma=4$ are valid deterministic bounds. The same
lazy-refresh construction gives exact mixing times $1$, $100$, and $334$,
with a common uniform stationary distribution and the same finite-sum
objective.

We compare the mixing-aware clipped method, the same method without clipping,
the clipped oblivious variant, single-sample \basevar{} that treats the stream
as independent, and projected proximal SGD. Independent i.i.d. paths with
seeds $100$--$104$ select the interior grid minimizers $\rho_0=0.3$ and
proximal step multiplier $0.3$. Report seeds $0$--$9$ are never used for
calibration. Every entry in \cref{tab:composite} evaluates the last recorded
iterate at or before the common budget of $90{,}000$ consumed states.

\begin{figure}[t]
\centering
\includegraphics[width=\linewidth]{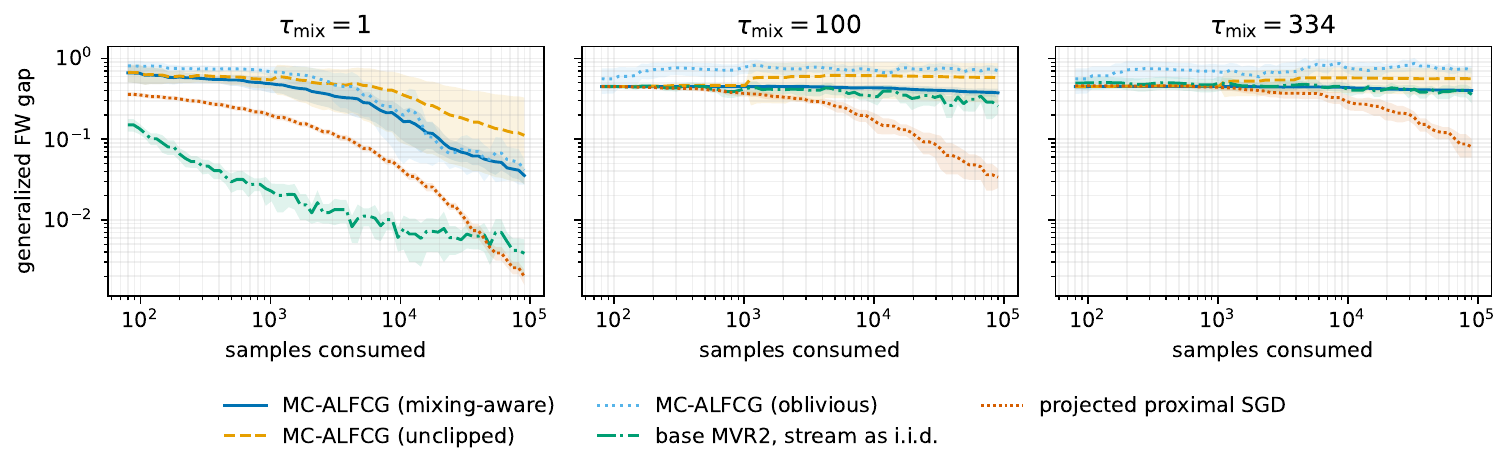}
\caption{Generalized Frank--Wolfe gap for the nonconvex composite instance
\cref{eq:composite-instance}. All methods are evaluated against consumed
Markov states. Lines and bands show means and one standard deviation over ten
paired seeds.}
\label{fig:composite-gap}
\end{figure}

\begin{table}[t]
\centering
\caption{Budget-matched generalized Frank--Wolfe gap at $90{,}000$ consumed
states for the nonconvex composite instance (mean $\pm$ standard deviation,
ten paired seeds). The common initial gap is $0.4481$.}
\label{tab:composite}
\small
\ifdefined\JOTAmode
\resizebox{\linewidth}{!}{%
\fi
\begin{tabular}{lccc}
\toprule
 & $\tmix=1$ & $\tmix=100$ & $\tmix=334$ \\
\midrule
\algname{} (mixing-aware) & $0.0355\pm0.0082$ & $0.3751\pm0.0481$ & $0.4001\pm0.0203$ \\
\algname{} (unclipped) & $0.1116\pm0.2166$ & $0.5783\pm0.2984$ & $0.5606\pm0.2023$ \\
\algname{} (oblivious) & $0.0416\pm0.0143$ & $0.7082\pm0.0659$ & $0.7272\pm0.0778$ \\
\basevar{}, stream as i.i.d. & $0.0038\pm0.0019$ & $0.2569\pm0.0509$ & $0.3579\pm0.0774$ \\
projected proximal SGD & $0.0020\pm0.0004$ & $0.0338\pm0.0091$ & $0.0799\pm0.0207$ \\
\bottomrule
\end{tabular}
\ifdefined\JOTAmode
}%
\fi
\end{table}

At the two dependent kernels, clipping lowers the gap in all ten paired
runs relative to the otherwise identical unclipped variant
(\cref{fig:composite-gap,tab:composite}). Along the unclipped trajectories,
the pre-clipping estimator norm exceeds $\Ghat=2$ on $3.40\%$ of iterations
on average and reaches $252.2$. The clipped trajectories activate the
safeguard on only $0.057\%$ of iterations on average. The two frequencies are
not directly interchangeable because clipping changes subsequent iterates.
The clipped method's paired deterioration relative to its own i.i.d. gap is
$11.1\times$ and $11.9\times$, compared with $76.7\times$ and $111.1\times$
for the naive base estimator. Nevertheless, projected proximal SGD has the
smallest absolute gap in every column. The experiment therefore supports the
clipping-stability and relative-dependence interpretations, not universal
empirical superiority. Twenty-nine of the thirty clipped runs finish below
the common initial gap. The retained exception is seed $7$ at
$\tmix=100$, whose terminal gap is $0.477$.

\end{document}